\documentclass[a4paper,UKenglish,cleveref, autoref, thm-restate]{lipics-v2021}

\pdfoutput=1 
\hideLIPIcs  

\title{On Euler Paths and the Maximum Degree Growth of Iterated Higher Order Line Graphs} 

\titlerunning{Euler Paths and the Maximum Degree Growth of Higher Order Line Graphs} 

\author{Aryan Sanghi}{Indian Institute of Technology Kharagpur, India \and \url{https://dragmonx.github.io} }{aryansanghi2004@gmail.com}{}{}

\author{Anubhav Dhar}{Max Planck Institute for Informatics, Germany \and \url{https://anubhavdhar.github.io/}}{anubhavldhar@gmail.com}{https://orcid.org/0009-0006-5922-8300}{}

\author{Sudeshna Kolay}{Indian Institute of Technology Kharagpur, India \and \url{https://cse.iitkgp.ac.in/~skolay/} }{skolay@cse.iitkgp.ac.in}{https://orcid.org/0000-0002-2975-4856}{}

\authorrunning{A. Sanghi, A. Dhar, S. Kolay} 



\category{} 

\relatedversion{} 

\nolinenumbers 

\usepackage{longtable}

\definecolor{thechosenone}{rgb}{0,0.1,0.5}

\usepackage{longtable}
\usepackage{hyperref}
\usepackage{amssymb,amsmath}
\usepackage{color}
\usepackage{makecell}
\usepackage{mathtools}
\usepackage{bbm}
\usepackage{setspace}

\usepackage{tikz}
\usetikzlibrary{arrows,positioning}

\newcommand{\defproblem}[3]{
  \vspace{1mm}
\begin{center}
\begin{nolinenumbers}
\noindent\fbox{

  \begin{minipage}{0.95\textwidth}
  \begin{tabular*}{\textwidth}{@{\extracolsep{\fill}}l} \textsc{\underline{#1}} \\ \end{tabular*}\vspace{1ex}
  {\bf{Input:}} #2  \\
  {\bf{Question:}} #3
  \end{minipage}

  }
  \end{nolinenumbers}
\end{center}
  \vspace{1mm}
}

\newdimen\prevdp
\def\leftlabel#1{\noalign{\prevdp=\prevdepth
   \kern-\prevdp\nointerlineskip\vbox to0pt{\vss\hbox{\ensuremath{#1}}}\kern\prevdp}}

\usepackage{url}

\usepackage{enumerate}

\usepackage{xspace}

\let\oldlambda\lambda
\renewcommand{\lambda}{\ensuremath{\oldlambda}\xspace}
\let\oldalpha\alpha
\renewcommand{\alpha}{\ensuremath{\oldalpha}\xspace}
\let\oldDelta\Delta
\renewcommand{\Delta}{\ensuremath{\oldDelta}\xspace}

\newcommand{\OO}{\ensuremath{\mathcal O}\xspace}

\usepackage{nicefrac}

\newcommand{\ignore}[1]{}

\renewcommand{\leq}{\leqslant}
\renewcommand{\geq}{\geqslant}
\renewcommand{\ge}{\geqslant}
\renewcommand{\le}{\leqslant}

\usepackage{enumitem}
\setlist[enumerate]{labelwidth=!, labelindent=0pt}

\usepackage{rotating}
\usepackage{amssymb}
\usepackage{latexsym}
\usepackage{epsfig}
\usepackage{xcolor}
\usepackage{url}
\usepackage{caption}
\usepackage{subcaption}
\usepackage{array}
\usepackage{multirow}
\usepackage{enumerate}
\usepackage{pdflscape}
\usepackage{amsmath,graphicx,algorithm,algpseudocode,amsfonts}
\usepackage{epstopdf}
\usepackage{float}
\floatstyle{ruled}
\newfloat{Algorithms}{!thb}{lop}
\floatname{Algorithms}{Algorithm}

\allowdisplaybreaks[1]

\algnewcommand\algorithmicinput{\textbf{Input:}}
\algnewcommand\INPUT{\item[\algorithmicinput]}
\algnewcommand\algorithmicoutput{\textbf{Output:}}
\algnewcommand\OUTPUT{\item[\algorithmicoutput]}
\algnewcommand{\LineComment}[1]{\State \(\triangleright\) #1}

\usepackage{pifont}

\usepackage{mathtools}

\usepackage{algorithm}
\usepackage{algpseudocode}


\usepackage{comment}

\usepackage{cleveref}
\definecolor{cblue}{RGB}{0, 0, 128}
\crefname{theorem}{Theorem}{\bf Theorems}
\crefname{observation}{Observation}{\bf Observations}
\crefname{corollary}{Corollary}{\bf Corollary}
\crefname{lemma}{Lemma}{\bf Lemmata}
\crefname{corollary}{Corollary}{\bf Corollaries}
\crefname{proposition}{Proposition}{\bf Propositions}
\crefname{definition}{Definition}{\bf Definitions}
\crefname{claim}{Claim}{\bf Claims}
\crefname{reductionrule}{Reduction rule}{\bf Reduction rules}
\usepackage{color,soul} 

\newcommand{\LGEPI}{\textsc{LGEPI}}
\newcommand{\MDGPI}{\mathrm{MDGPI}}
\newcommand{\dgc}{\mathsf{dgc}}
\usepackage{todonotes}

\begin{document}

\maketitle
\begin{abstract}
Given a simple graph $G$, its \emph{line graph}, denoted by $L(G)$, is obtained by representing each edge of $G$ as a vertex, with two vertices in $L(G)$ adjacent whenever the corresponding edges in $G$ share a common endpoint. By applying the line graph operation repeatedly, we obtain \emph{higher order line graphs}, denoted by $L^{r}(G)$. In other words, $L^{0}(G) = G$, and for any integer $r \ge 1$, $L^{r}(G) = L(L^{r-1}(G))$.

In this work, we conduct a structural analysis of higher order line graphs. We consider two problems that are solvable on graphs in polynomial time: (i) finding whether an Euler path exists, (ii) finding the maximum degree of a graph. Given a graph $G$ on $n$ vertices, we wish to efficiently find out (i) if $L^k(G)$ has an Euler path, (ii) the value of $\Delta(L^k(G))$. Note that the size of a higher order line graph could be much larger than that of $G$. 

For the first question, we show that for a graph $G$ with $n$ vertices and $m$ edges the largest $k$ where $L^k(G)$ has an Euler path satisfies $k = \mathcal O(nm)$. We also design an $\mathcal{O}(n^2m)$-time algorithm to output all $k$ such that $L^k(G)$ has an Euler path. 

For the second question, we study the growth of maximum degree of $L^k(G)$, $k \ge 0$, for a graph $G$. It is easy to calculate $\Delta(L^k(G))$ when $G$ is a path, cycle or a claw. Any other connected graph is called a \emph{prolific graph}~\cite{caro2022index} and we denote the set of all prolific graphs by $\mathcal G$. We extend the works of Hartke and Higgins~\cite{stephen1999growth} to show that for any prolific graph $G$, there exists a constant rational number $\dgc(G)$ and an integer $k_0$ such that for all $k \ge k_0$, $\Delta(L^k(G)) = \dgc(G) \cdot 2^{k-4} + 2$. We show that $\{\dgc(G) \mid G \in \mathcal G\}$ has first, second, third, fourth and fifth minimums, namely, $c_1 = 3$, $c_2 = 4$, $c_3 = 5.5$, $c_4 = 6$ and $c_5=7$; the third minimum stands out surprisingly from the other four. Moreover, for $i \in \{1, 2, 3, 4\}$, we provide a complete characterization of $\mathcal G_i = \{\dgc(G) = c_i \mid G \in \mathcal G \}$. Apart from this, we show that the set $\{\dgc(G) \mid G \in \mathcal G, 7 < \dgc(G) < 8\}$ is countably infinite.

The results in this paper contribute to a deeper understanding of how efficiently checkable graph properties evolve under repeated line graph transformations.
\end{abstract}

\section{Introduction}

The study of higher order line graphs forms a rich and classical area of graph theory with deep connections to structure theory, forbidden subgraph characterization, and extremal graph properties. For a simple connected graph $G$, its \emph{line graph} $L(G)$ is defined as the graph whose vertices correspond to the edges of $G$, with two vertices adjacent in $L(G)$ if and only if their corresponding edges in $G$ share a common endpoint. The $r$-th order line graph of $G$, denoted as $L^r(G)$, is defined inductively as $L^0(G) = G$ and $L^r(G) = L(L^{r-1}(G))$ for $r \ge 1$.

Higher order line graphs have been investigated from multiple viewpoints—combinatorial, algorithmic, and structural. Early foundational work by Hemminger and Beineke~\cite{hemminger1978line} characterized all line graphs in terms of forbidden induced subgraphs, leading to the celebrated list of nine minimal forbidden configurations for the class of line graphs~\cite{beineke1970characterization}. Subsequent studies refined these results and examined their implications for higher order line graph constructions and their structural invariants.


It is clear from the definition of line graphs and higher order line graphs, that it is possible that for a graph $G$ on $n$ vertices, the $r$-th order line graph $L^{r}(G)$ has as many as $2^{\mathcal O (k^2)} \cdot n^{\mathcal O(k)}$ vertices (for example the complete graph on $n$ vertices). Thus, it is interesting to see the evolution of efficiently checkable graph properties in the infinite sequence of graphs $\{G,L(G),L^2(G),\ldots,L^i(G),\ldots\}$ created by iterated applications of the line graph operation on a graph $G$. In this paper, we are particularly interested in two polynomial-time checkable graph properties: (i) existence of an Euler path in a graph, (ii) calculating the maximum degree of a graph. We wish to analyze the two questions with respect to the number $n$ of vertices in the initial graph $G$.


\subsection{Previous Results}
Beineke~\cite{beineke1970characterization} proved that a graph $G$ is a line graph if and only if it contains none of nine specific graphs as induced subgraphs, which includes the \emph{claw} as one of the forbidden induced subgraphs. The \emph{claw} ($K_{1,3}$) is the graph of four vertices, say $a, b, c, d$, and three edges $ab, ac, ad$. Šoltés~\cite{soltes1994forbidden} later reduced this forbidden list to seven graphs, providing a more succinct formulation of the same characterization. Hemminger and Beineke~\cite{hemminger1978line} showed that for any subgraph $H$ of $G$, the corresponding $L(H)$ is an induced subgraph of $L(G)$, which provides an inductive mechanism for understanding the structure of higher order line graphs. The recent work by Sanghi et. al~\cite{sanghi2024forbidden} extends these ideas to higher order iterations, providing a forbidden subgraph-based framework for identifying when a graph $G$ can be expressed as $L^k(H)$ for some $H$. They introduce the concept of a \emph{pure induced line subgraph}, a subgraph $G'$ of a line graph $G = L(H)$ satisfying $G' = L(H')$ for an induced subgraph $H'$ of $H$. This characterization was further generalized: $G = L^n(H)$, $n \ge 2$ for some $H$, if $L^{-i}(G)=L^{n-i}(H)$ is claw-free (i.e. does not contain a claw as an induced subgraph) for all $i$, $1 \le i \le n-1$ and $G$ avoids a prescribed finite family of forbidden induced subgraphs.

An Euler circuit of a graph is a walk in the graph starting and ending at the same vertex, and traversing every edge of the graph exactly once. A graph is said to be Eulerian if it has an Euler circuit. Prisner~\cite{Erich2000Euler} studied whether higher order line graphs are Eulerian. The work analyzed the values of $k$ such that $L^k(G)$ is Eulerian. He established that:

\begin{proposition}[Prisner~\cite{Erich2000Euler}]
A connected graph $G$ has some iterated line graph $L^t(G)$ that is Eulerian if and only if $G$ is a path, or all its vertices have degrees of the same parity, or $G$ is bipartite with every edge joining vertices of opposite parity.
\end{proposition} 

This work also showed that if such a $t$ exists, then $L^r(G)$ is also Eulerian, for all $r \ge t$.

A Hamiltonian cycle of a graph $G$ is a simple cycle containing all vertices of $G$. A graph is said to be Hamiltonian if it has a Hamiltonian cycle. It is known that higher order line graphs eventually become Hamiltonian~\cite{harary1991graph}.

\begin{proposition}\label{prop:eventually-hamiltonian}
    
    For a connected graph $G$ of $n$ vertices which is not isomorphic to a path, $L^k(G)$ is Hamiltonian for all $k \ge n - 3$.    
\end{proposition}

Parallel to these structural developments, Hartke and Higgins~\cite{stephen1999growth} analyzed the asymptotic behavior of $\Delta(L^k(G))$ (where $\Delta(H)$ denotes the maximum degree of the graph $H$), defining the \textit{Maximum Degree Growth Property (MDGP)}:

\begin{definition}[Hartke and Higgins~\cite{stephen1999growth}]
    A graph $G$ is said to satisfy Maximum Degree Growth Property (MDGP) if $\Delta(L(G)) = 2\Delta(G) - 2$.
\end{definition}

They show that for every connected graph $G$ which is not a path, there exists a finite integer $K$ (the \textit{MDGP Index}) such that MDGP holds for all $L^k(G)$ with $k \ge K$, and provided the recurrence
\(
\Delta(L^k(G)) = (\Delta(L^K(G)) - 2)\,2^{k-K} + 2
\) (see Definition~\ref{def:MDGPI}).
This result implies that the maximum degree of iterated line graphs stabilizes into a predictable exponential growth regime after finitely many iterations.

\subsection{Our Contributions}
In this paper, we first extend the results of Prisner~\cite{Erich2000Euler} on Euler circuits to Euler paths. An Euler path is a walk that starts and ends at two distinct vertices of the given graph, and traverses each edge exactly once. The tools required to characterise the existence of an Euler path in a higher order line graph $L^k(G), k\geq 1$ are very different from those used in~\cite{Erich2000Euler}. We show that for a graph $G$ with $n$ vertices and $m$ edges the largest $k$
where $L^k (G)$ has an Euler path satisfies $k = \OO(nm)$. We also design an $\OO(n^2m)$-time algorithm to
output all $k$ such that $L^k (G)$ has an Euler path. Our results rely on the analysis of local relations such as the characterisation of the initial graph $G$ when one of $L(G)$ or $L^2(G)$ does not have any Euler path. Note that our results are in contrast to those derived in~\cite{Erich2000Euler}, as we show that eventually higher order line graphs will not have Euler paths.

The second question we address in this paper extends the work of Hartke and Higgins~\cite{stephen1999growth}. We study the growth of maximum degree of $L^k(G)$, $k \ge 0$, for a graph $G$. It is easy to calculate $\Delta(L^k(G))$ when $G$ is a path, cycle or a claw. Any other connected graph is called a \emph{prolific graph}~\cite{caro2022index} and we denote the set of all prolific graphs by $\mathcal G$. We show that for any prolific graph $G$, there exists a constant rational number $\dgc(G)$ and an integer $k_0$ such that for all $k \ge k_0$, $\Delta(L^k(G)) = \dgc(G) \cdot 2^{k-4} + 2$. We show that $\{\dgc(G) \mid G \in \mathcal G\}$ has first, second, third, fourth, and fifth minimums, namely, $c_1 = 3$, $c_2 = 4$, $c_3 = 5.5$, $c_4 = 6$ and $c_5=7$; the third minimum stands out surprisingly from the other four. Moreover, for $i \in \{1, 2, 3, 4\}$, we provide a complete characterization of $\mathcal G_i = \{\dgc(G) = c_i \mid G \in \mathcal G \}$. Apart from this, we show that the set $\{\dgc(G) \mid G \in \mathcal G, 7 < \dgc(G) < 8\}$ is countably infinite. This landscape is depicted in Figure~\ref{fig:dgc-landscape}.

\begin{figure}
    \centering
    \includegraphics[width=\linewidth]{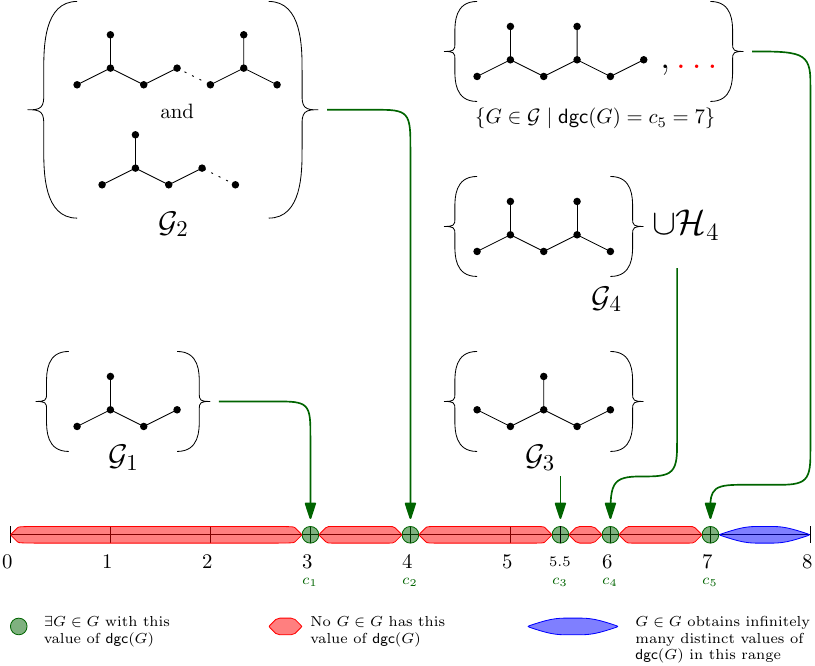}
    \caption{Possible values of $\dgc(G)$. Class $\mathcal H_4$ is defined later (see Definition~\ref{def:H_4}). The graph class corresponding to $\dgc(G) = c_5$ is not fully characterised in our work.}
    \label{fig:dgc-landscape}
\end{figure}
\section {Preliminaries}

\subparagraph*{Notations.} All graphs in this paper are undirected, unweighted, connected and simple. For a graph $G$, we denote the set of vertices as $V(G)$ and the set of edges as $E(G)$. The set $N_G(v) = \{u \mid uv \in E(G)\}$ denotes the open neighbourhood of vertex $v$ in the graph $G$. The degree of a vertex $v$ in $G$ is $d_G(v) = |N_G(v)|$. For a set of vertices $S \subseteq V$, the subgraph of $G$ induced by $S$ is denoted by $G[S]$. A graph $H$ is a \emph{subgraph} of a graph $G$ if there exists a mapping $\phi : V(H) \to V(G)$ such that for all $\{u,v\} \in E(H)$, $\{\phi(u), \phi(v)\} \in E(G)$; we denote this by $H \subseteq G$. The path on $n$ vertices is denoted by $P_n$ and the cycle on $n$ vertices is denoted by $C_n$. $K_n$ denotes the complete graph on $n$ vertices, and $K_{a,b}$ denotes the complete bipartite graph with $a$ vertices in one part and $b$ vertices in the other. Graphs $K_{1, r}$ for $r \ge 1$ are called \emph{star} graphs, and the graph $K_{1.3}$ is called a \emph{claw}. 

\subparagraph*{Line graphs.} For a graph $G$, the line graph $L(G)$ denotes the graph with $\{v_e \mid e \in E(G)\}$ as the vertex set and $\{v_{e_1}v_{e_2} \mid e_1, e_2 \in E(G), e_1\text{ and } e_2 \text{ share a vertex}\}$ as the edge set. We say that the edge $e \in E(G)$ and the vertex $v_e \in V(L(G))$ \emph{correspond} to each other. We define $L^0(G) = G$ and $L^r(G) = L(L^{r-1}(G))$, for all $r \ge 1$. We state some well-known results for line graphs~\cite{stephen1999growth}.

\begin{proposition}\label{Deguv}
    For a graph $G$, let $uv\in E(G)$. Then, in $L(G)$, the degree of the vertex corresponding to the edge $uv$ is $d_G(u)+d_G(v)-2$.
\end{proposition}

\begin{proposition}[Induced subgraph preservation]\label{subindsub}
If $H$ is a subgraph of $G$, then $L(H)$ is an induced subgraph of $L(G)$.
\end{proposition}

For graphs $H$ and $G = L(H)$, we say that $H$ is a \emph{preimage} of $G$ under the line graph operation. It is well known that for any graph $G$ not isomorphic to a triangle, either no preimage of $G$ exists under the line graph operation, or it is unique~\cite{whitney1992congruent}.
\begin{proposition}
    For non-isomorphic graphs $G_1$, $G_2$, we have $L(G_1) = L(G_2)$ if and only if $G_1$ is isomorphic to $C_3$ (a triangle) and $G_2$ is isomorphic to $K_{1,3}$ (a claw), or vice versa. 
\end{proposition}

\section{Eulerian Path on Higher Order Line Graphs}


For a connected graph $G$, recall that an Euler path (also known as an Eulerian path or an Eulerian trail) is a walk that starts and ends at distinct vertices, and traverses every edge of the graph exactly once. Note that we do not consider Euler circuits as Euler paths. A well known necessary and sufficient condition for the existence of Euler paths is stated below~\cite{Euler2010Path}.
\begin{proposition}\label{prop:two-odd}
    A connected graph has an Euler path if and only if exactly two of its vertices have an odd degree.
\end{proposition}

Deciding whether $L^k(G)$ has an Euler path can be naively done by checking the number of odd degree vertices of $L^k(G)$. However, such an algorithm would take time exponential in $k$ and size of $G$ to terminate. In this section, we look into various structural properties of $L^k(G)$, aiding us to solve this problem more efficiently. We consider an exhaustive casework which analyzes all possible local structures of the graph $G$ and deduce the values of $k$ for which $L^k(G)$ has an Euler path. This not only provides us with an exact classification of graphs where certain higher order line graphs admit Euler paths, but also results in a polynomial time algorithm deciding whether $L^k(G)$ has an Euler path. In fact, our algorithm is stronger: For an input graph $G$ with $n$ vertices and $m$ edges, in $\OO(n^2m)$ time, the algorithm outputs all $k$ such that $L^k (G)$ has an Euler path. Notice that the higher order line graphs eventually stop admitting Euler paths; such a result has an opposing flavour to Proposition~\ref{prop:eventually-hamiltonian}.


\subsection{Definitions and tools}\label{subsec:toolsEuler}

We start by laying down some definitions and observations relevant to the study of the existence of Euler paths in higher order line graphs.

\begin{definition}
     For a graph $G$ and a set $S\in V(G)$, let  $N_G(S) = \bigcup_{v \in S} N_G(v)$. Let $ov(G)$ denote the set of all vertices of a graph $G$ that have an odd degree and let $ev(G)$ denote the set of all vertices of a graph $G$ with an even degree. 
\end{definition}

We introduce the notion of an edge being \emph{critical} depending on the degrees of its endpoints.

\begin{definition}
    We call an edge $v_1v_2\in E(G)$ critical if either $v_1\in ov(G), v_2\in ev(G)$ or $v_1\in ev(G), v_2\in ov(G)$. Otherwise, the edge is said to be non-critical. 
\end{definition}

We use the following notation that associates edges of a graph to vertices in its line graph.

\begin{definition}
    For a vertex $v\in V(L(G))$, let $e_v\in E(G)$ be its corresponding edge in $G$.
\end{definition}

Proposition~\ref{Deguv} directly implies the following relation between critical edges of $G$ and the odd degree vertices of $L(G)$. 

\begin{observation}\label{obs:odd-critical}
For a graph $G$, a vertex $v\in V(L(G))$ in its line graph has an odd degree if and only if $e_v\in E(G)$ is a critical edge. Moreover, if $v_1v_2\in E(L(G))$ is a critical edge, then $e_{v_1}\in E(G)$ and $e_{v_2}\in E(G)$ are incident on a common vertex in $G$, and exactly one among $e_{v_1}$ and $e_{v_2}$ is a critical edge.
\end{observation}

The next three results are about relations between critical edges of $G$ and those of $L(G)$.
\begin{observation}\label{obs:critical-and-non-critical}
    In a graph $G$, let $uv\in E(G)$ be a critical edge and $uw\in E(G)$ be a non-critical edge. Let $uv$ correspond to the vertex $x\in V(L(G))$ and $uw$ correspond to the vertex $y\in V(L(G))$. Then, $xy\in E(L(G))$ is a critical edge.
\end{observation}
\begin{proof}
    Since $uv\in E(G)$ is a critical edge, $x\in ov(L(G))$ by Observation~\ref{obs:odd-critical}. Similarly, since $uv\in E(G)$ is a non-critical edge, $y\in ev(L(G))$. Since $uv$ and $uw$ and incident on same vertex $u$, vertices $x$ and $y$ are adjacent in $L(G)$, and since one has degree odd and other even, $xy\in E(L(G))$ is a critical edge.
\end{proof}

\begin{lemma}\label{lemma:always-critical}
    Let a vertex $v\in V(G)$ have both a critical edge, say $vv_c\in E(G)$, and a non-critical edge, say $vv_{nc}\in E(G)$, incident on it. Suppose $v$ has an edge $vu\in E(G)$ incident on it other than $vv_c$ and $vv_{nc}$. Then the vertex $x \in V(L(G))$ which corresponds to $vu \in E(G)$ has some critical edge incident on it in $L(G)$.
\end{lemma}

\begin{proof}
    Let the vertices corresponding to edges $vv_c$ and $vv_{nc}$ be $a \in V(L(G))$ and $b \in V(L(G))$ respectively. Now, $a\in ov(L(G))$ and $b\in ev(L(G))$ by Observation~\ref{obs:odd-critical}. Since $vv_c$ and $vu$ are incident of same vertex $v$, $xa\in E(L(G))$. Similarly, $xb\in E(L(G))$. If the edge $vu$ is a critical edge, then $x \in ov(L(G))$ (by Observation~\ref{obs:odd-critical}) implying that $xb$ is a critical edge in $L(G)$. Otherwise, if the edge $vu$ is a non-critical edge, then $x \in ev(L(G))$ (by Observation~\ref{obs:odd-critical}) implying that $xa$ is a critical edge in $L(G)$. 
\end{proof}

\begin{lemma}\label{lem:claw-critical}
    For a vertex $v\in V(G)$, let there exist three edges $vu_1$, $vu_2$ and $vu$ such that $vu_1$ and $vu_2$ are critical ($vu$ may or may not be critical). Suppose there exists an edge $e$ connecting any two of $u_1$, $u_2$ and $u_3$ (i.e. $v$, $u_1$, $u_2$ and $u$ don't induce a claw in $G$). Then the vertex $x\in V(L(G))$ corresponding to $e \in  E(G)$, has a critical edge incident on it.
\end{lemma}
\begin{proof}
    If $e = u_1u_2$, then $e$ is a non-critical edge. Moreover $vu_1$ is a critical edge. So, by Observation~\ref{obs:critical-and-non-critical}, there will be a critical edge incident on $x \in L(G)$. 

    Otherwise, $e \in \{uu_1, uu_2\}$, without loss of generality, assume $e = uu_1$. If $vu$ is non-critical, then $e = uu_1$ will be critical, as parities of the degrees of $v$ and $u$ are different from the parities of degrees of $u_1$. Similarly, if $vu$ is critical, then $e = uu_1$ will be non-critical. In both case, $u$ has both a critical and a non-critical edge incident on it. So, by Observation~\ref{obs:critical-and-non-critical}, there will be a critical edge incident on $x \in L(G)$. 
\end{proof}
We are now ready to analyze the structural properties of higher order line graphs relevant to the existence of Euler paths in them.

\subsection{Structural results}\label{subsec:structureEuler}

We start by enumerating and analyzing the possible cases of graphs $G$, where $L(G)$ does not have an Euler path but $L^2(G)$ does. 

Since $L^2(G)$ has an Euler 
path, it has exactly two odd degree vertices, i.e. $|ov(L^2(G))|=2$. By Observation~\ref{obs:odd-critical}, $L(G)$ has exactly two critical edges. Moreover, since $L(G)$ does not have an Euler path, it must have more than two odd degree vertices. Thus, by handshaking Lemma~\cite{gunderson2010handbook}, it must contain at least four odd degree vertices, i.e. $|ov(L(G))|\geq 4$.

\begin{observation}\label{obs:4-ov}
    If $L(G)$ has no Euler path but $L^2(G)$ does, then $|ov(L^2(G))| = 2$ and $|ov(L(G))|\geq 4$.
\end{observation}

Since $|ov(L^2(G))|=2$, by Observation~\ref{obs:odd-critical}, there must be exactly two critical edges in $L(G)$. These two critical edges can have the following cases (Figure~\ref{fig:cases-of-L(G)}).

\begin{itemize}
    \item Both critical edges of $L(G)$ share an odd degree vertex as an endpoint (Lemma~\ref{lem:share-odd}).
    \item Both critical edges of $L(G)$ share an even degree vertex as an endpoint (Lemma~\ref{lem:share-even}).
    \item The critical edges of $L(G)$ do not have a common endpoint (Lemma~\ref{lem:matching-edges}).
\end{itemize}

\begin{figure}[ht!]
    \centering
    \includegraphics[width=\textwidth]{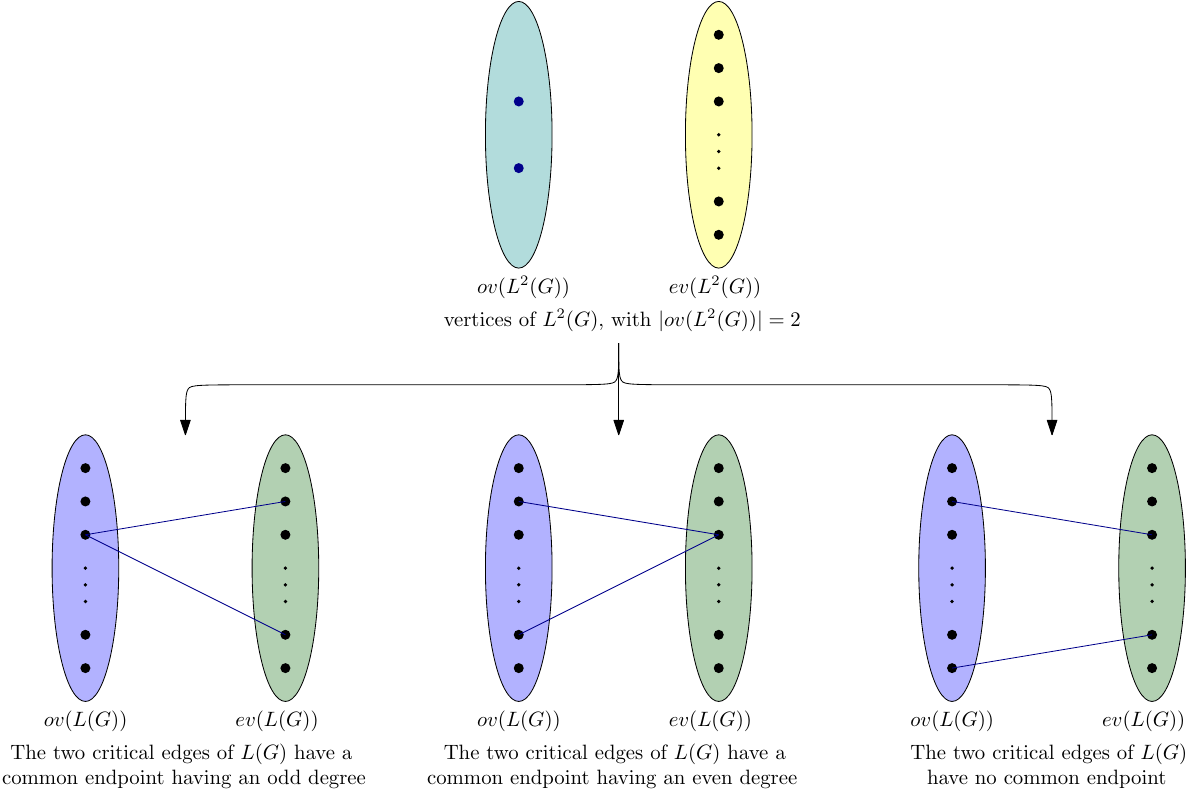}
    \caption{Cases of the two critical edges in $L(G)$.}\label{fig:cases-of-L(G)}
\end{figure}

In the next three results, we analyze each of these cases of the critical edges of $L(G)$ and deduce properties of $G$. We provide a proof sketch for Lemma~\ref{lem:share-odd}. Since the proofs of the other two lemmata are similar, we move it to the Appendix.

\begin{lemma}\label{lem:share-odd}
    Let $G$ be a graph such that $L(G)$ does not have an Euler path, but $L^2(G)$ does.  If both critical edges of $L(G)$ share an odd degree vertex as an endpoint then there is no graph $H$ such that $L(H) = G$.
\end{lemma}



\begin{proof}
    Let the two critical edges $L(G)$ be $ax$ and $ay$ where $a\in ov(L(G))$ and $x,y\in ev(L(G))$ (Figure~\ref{fig:common-odd}). Since $ax \in E(L(G))$, the edges $e_a \in E(G)$ and $e_x \in E(G)$ of $G$ must have a common endpoint. Similarly, $e_a \in E(G)$ and $e_y \in E(G)$ have a common endpoint. Moreover, as $a\in ov(L(G))$ and $x,y\in ev(L(G))$, $e_a$ is a critical edge in $G$ while $e_x$ and $e_y$ are non-critical edges. This gives us the following cases.
    
    \begin{figure}[ht!]
        \centering
        \includegraphics[width=\textwidth]{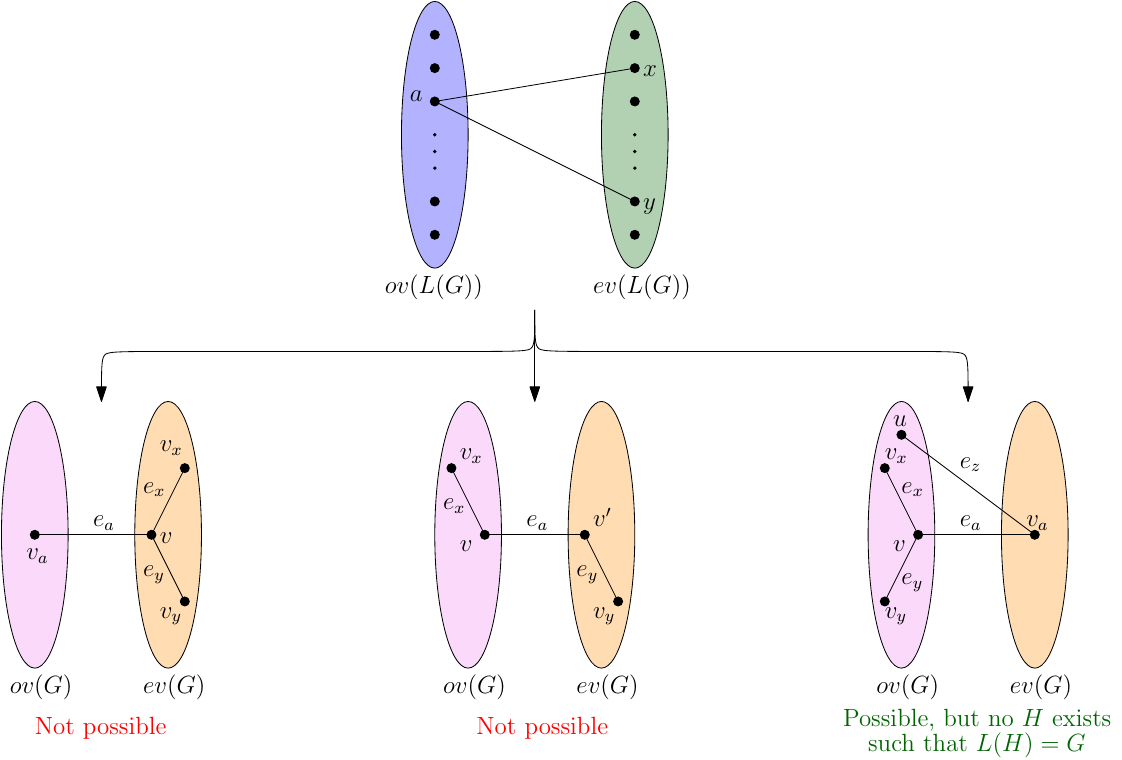}
        \caption{Cases of $G$ when both critical edges of $L(G)$ share an odd degree vertex as an endpoint.}\label{fig:common-odd}
    \end{figure}

    \subparagraph*{Case I: $e_a, e_x$ and $e_y$ are incident on a common vertex $v\in ev(G)$, having an even degree.} Let the other endpoints of $e_a$, $e_x$ and $e_y$ be $v_a$, $v_x$, and $v_y$ respectively. Recall that $e_a$ is critical in $G$ while $e_x$ and $e_y$ are not; and since $v \in ev(G)$ we must have $v_x$, $v_y \in ev(G)$ and $v_a \in ov(G)$. $v$ must have at least one neighbour other than $v_a, v_x$ and $v_y$, as $v$ has an even degree. Let $u \in N_G(v) \setminus \{v_a, v_x, v_y\}$ be such a neighbour. Let $z \in V(L(G))$ be the vertex in $L(G)$ corresponding to the edge $uv \in E(G)$. Notice that the vertex $v$ has a critical edge $e_a$ and a non-critical edge $e_x$ incident on it. So, by Lemma~\ref{lemma:always-critical}, $z$ always has a critical edge incident on it in $L(G)$ --- a contradiction as $L(G)$ has exactly two critical edges $ax$ and $ay$. So, this case is not possible.
    
    \subparagraph*{Case II: $e_a$ and $e_x$ are incident on a common vertex $v\in ov(G)$ having an odd degree; and $e_a$ and $e_y$ are incident on a common vertex $v' \in ev(G)$ having an even degree.} Let the other endpoints of $e_x$ and $e_y$ be $v_x$ and $v_y$ respectively. Note that $v_x \in ov(G)$ and $v_y \in ev(G)$ as $e_x$ and $e_y$ are non-critical edges of $G$. $v$ must have at least one neighbour other than $v'$ and $v_x$, as $v$ has an odd degree. Let $u \in N_G(v) \setminus \{v', v_x\}$ be such a neighbour. Let $z \in V(L(G))$ be the vertex in $L(G)$ corresponding to the edge $uv \in E(G)$. Notice that the vertex $v$ has a critical edge $e_a$ and a non-critical edge $e_x$ incident on it. So, by Lemma~\ref{lemma:always-critical}, $z$ always has a critical edge incident on it in $L(G)$ --- a contradiction as $L(G)$ has exactly two critical edges $ax$ and $ay$. So, this case is not possible.
    
    \subparagraph*{Case III: $e_a, e_x$ and $e_y$ are incident on a common vertex $v\in ov(G)$, having an odd degree.} We show that this case is possible, but there exists no graph $H$ such that $G=L(H)$. Let the endpoints of $e_a$, $e_x$ and $e_y$ other than $v$ be $v_a$, $v_x$ and $v_y$ respectively. Then, $v_a \in ev(G)$ while $v_x,v_y \in ov(G)$ as $e_a$ is a critical edge and $e_x$ and $e_y$ are non-critical edges with $v\in ov(G)$. Now, $v_a$ must have a neighbour other than $v$ as $v_a$ has an even degree. Let $u \in N_{G}(v_a) \setminus\{v\}$ be such a neighbour. Let $z \in V(L(G))$ be the vertex in $L(G)$ corresponding to the edge $uv_a \in E(G)$.
    
    If $u \in ev(G)$, then $uv$ is a non-critical edge. So $az$ is a critical edge in $L(G)$ by Observation~\ref{obs:critical-and-non-critical} --- a contradiction as $L(G)$ has exactly two critical edges $ax$ and $ay$. Thus, we must have $u \in ov(G)$, implying $N_G(v_a) \subseteq ov(G)$. If $u$ has any neighbour other than $v_a$, say $u'$, then we must have $u'\in ev(G)$. This is because if $u'\in ov(G)$, then $uu'\in E(G)$ is a non-critical edge  $uv_a\in E(G)$ is a critical edge, both incident on $u$. So, by Observation~\ref{obs:critical-and-non-critical}, there will yet another critical edge in $L(G)$ other than $ax$ and $ay$ --- a contradiction. Therefore $N_G(u) \subseteq ev(G)$. Note that in $G$, both $e_a$ and $uv_a$ are critical edges. 
    
    $L(G)$ has at least four odd degree vertices since $L(G)$ does not have an Euler path (By Observation~\ref{obs:4-ov}). So, $G$ must have at least $4$ critical edges, i.e., there exists some critical edge $e'$ other than $e_a$ and $uv_a$. We now show that there exists an induced claw in $G$. We look into the following cases.
    
    \textit{Case (a): $d_G(v_a) > 2$.} There exists an edge $wv_a$ other than $e_a$ and $v_au$ be incident on $v_a$. Then, $w\in N_G(v_a) \subseteq ov(G)$. The vertices $v_a$, $v$, $u$ and $w$ induce a claw in $G$ as otherwise there will be yet another critical edge in $L(G)$ distinct from $ax$ and $ay$, by Lemma~\ref{lem:claw-critical} ---  a contradiction.

    \textit{Case (b): $d_G(u) > 1$.} There exists an edge $wu$ be incident on $u$ other than $uv_a$. Since $u$ has an odd degree, $u$ has at least one more neighbour other than $w$ and $v_a$. Let $u' \in N_G(u) \setminus \{w, v_a\}$ be such a neighbour. Then, $u',w\in N_G(u) \subseteq ev(G)$. Also, the vertices $u$, $v_a$, $w$ and $u'$ induce a claw in $G$ as otherwise there will be yet another critical edge in $L(G)$ other than $ax$ and $ay$ by Lemma~\ref{lem:claw-critical} --- a contradiction.

    \textit{Case (c): $d_G(v_a) = 2$ and $d_G(u) = 1$.} Consider a path from $v_a$ to either endpoint of $e'$; such a path must contain $e_a$. Therefore, the critical edge $e'$ must be incident on a vertex $v'$ in the connected component $C$ of $G\setminus\{e_a\}$ which includes $v$. Further without loss of generality, assume $v'\in V(C)$ is the nearest vertex to $v$ in $C$ with a critical edge incident on it. Let $e' = v'w$. Now, since $v'$ is connected to $v$ in $C$ and no other vertex in the shortest path from $v$ to $v'$ in $C$ has a critical edge incident on it (due to the choice of $v'$), there exists a non-critical edge $w'v'$, appearing in the shortest path from $v$ to $v'$ in $C$ (or the non critical edge $e_x$ if $v' = v$). So, $v'$ has a critical edge and a non-critical edge incident on it. So, by Observation~\ref{obs:critical-and-non-critical}, there will be yet another critical edge in $L(G)$ other than $ax$ and $ay$ --- a contradiction. So, this case is not possible.
    
    The cases I, II and III are exhaustive. So, $G$ must contain an induced claw. Hence, there is no graph $H$ such that $G=L(H)$~\cite{beineke1970characterization}, i.e. $G$ is not a line graph of any graph.
\end{proof}

We define the graph $EG_0$ as shown in Figure~\ref{fig:EG_0}. 

\begin{figure}[ht!]
    \centering
    \includegraphics[width=0.2\textwidth]{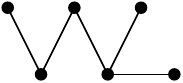}
    \caption{The graph $EG_0$}\label{fig:EG_0}
\end{figure}

\begin{lemma}\label{lem:share-even}
    Let $G$ be a graph such that $L(G)$ does not have an Euler path, but $L^2(G)$ does. If both critical edges of $L(G)$ share an even degree vertex as an endpoint, then there is no graph $H$, other than $EG_0$, such that $L(H) = G$.
\end{lemma}

\begin{proof} 
    Let the two critical edges be $ax$ and $ay$ where $a\in ev(L(G))$ and $x,y\in ov(L(G))$ (Figure~\ref{fig:share-even}). Observe that $e_a$ and $e_x$ will have a common endpoint in $G$, and so will $e_a$ and $e_y$. Moreover, $e_x$ and $e_y$ are critical edges whereas $e_a$ is a non-critical edge. The following cases arise:
    
    \begin{figure}[ht!]
        \centering
        \includegraphics[width=\textwidth]{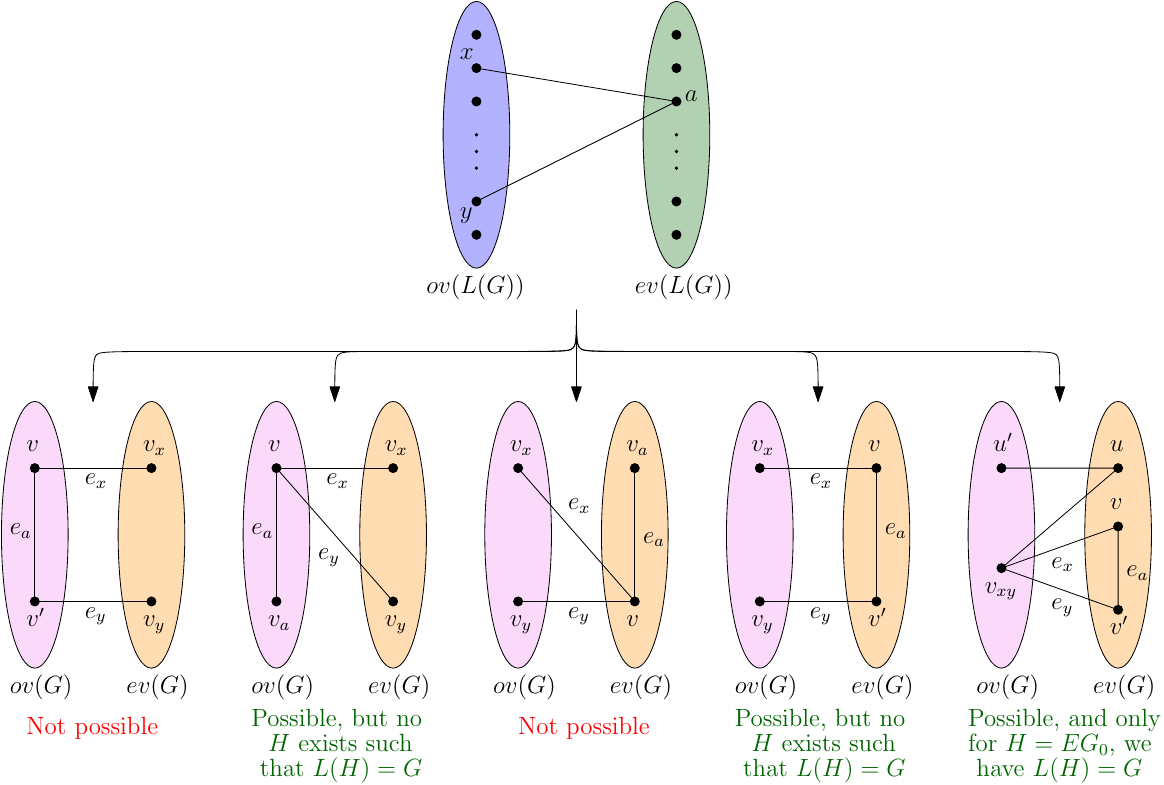}
        \caption{Cases of $G$ when both critical edges of $L(G)$ share an even degree vertex as an endpoint.}\label{fig:share-even}
    \end{figure}

    \subparagraph*{Case I: Both endpoints of $e_a$, say $v$ and $v'$, have an odd degree and $e_x$ is incident on $v$ and $e_y$ is incident on $v'$.} Let the other endpoints of $e_x$ and $e_y$ be $v_x$, $v_y$ respectively. Then, $v_x$, $v_y \in ev(G)$ as $e_x$ and $e_y$ are critical edges with $v$, $v'\in ov(G)$. Now, $v$ has at least one neighbour other then $v'$ and $v_x$, as $v$ has an odd degree. Let $u\in N_G(v)\setminus\{v',v_x\}$ be such a neighbour. Let $z \in V(L(G))$ be the vertex in $L(G)$ corresponding to the edge $uv \in E(G)$. Now, the vertex $v$ has a critical edge $e_x$ and a non-critical edge $e_a$ incident on it. So, by Lemma~\ref{lemma:always-critical}, $z$ has a critical edge incident on it in $L(G)$ --- a contradiction as $L(G)$ only has two critical edges $ax$ and $ay$. So, this case is not possible.
    
    \subparagraph*{Case II: Both endpoints of $e_a$ have an odd degree, and all three edges $e_a$, $e_x$ and $e_y$ are incident on a common vertex, say $v\in ov(G)$.} We show that this case is possible, but there exists no graph $H$ such that $G=L(H)$. Let the other endpoints of $e_x,e_a$ and $e_y$ be $v_x$, $v_a$ and $v_y$ respectively. Then, $v_x$, $v_y \in ev(G)$ as $e_x$ and $e_y$ are critical edges with $v\in ov(G)$. The vertices $v$, $v_x$, $v_a$ and $v_y$ induce a claw in $G$ as otherwise there will be yet another critical edge in $L(G)$ distinct from $ax$ and $ay$ by Lemma~\ref{lem:claw-critical} --- a contradiction. Therefore,  there is no graph $H$ such that $G=L(H)$~\cite{beineke1970characterization}, i.e. $G$ is not a line graph of any graph.
    
    \subparagraph*{Case III: $e_a$ has both endpoints of even degree, and all three edges $e_x$, $e_a$ and $e_v$ are incident on a common vertex, say $v\in ev(G)$.} Let the other endpoints of $e_a$, $e_x$ and $e_y$ be $v_a$, $v_x$ and $v_y$ respectively. Then $v_a\in ev(G)$ and $v_x$, $v_y \in ov(G)$ as $e_a$ is a non-critical edge and $e_x$ and $e_y$ are critical edges with $v\in ev(G)$. Now, $v$ must have a neighbour other than $v_a$, $v_x$ and $v_y$ as it has an even degree. Let $u\in N_G(v)\setminus\{v_a,v_x,v_y\}$ be such a neighbour. Let $z \in V(L(G))$ be the vertex in $L(G)$ corresponding to the edge $uv \in E(G)$. Now, vertex $v$ has a critical edge $e_x$ and a non-critical edge $e_a$ incident on it. So, by Lemma~\ref{lemma:always-critical}, $z$ always has a critical edge incident on it in $L(G)$ --- a contradiction as $L(G)$ only has two critical edges $ax$ and $ay$. So, this case is not possible.

    \subparagraph*{Case IV: $e_a$ has both endpoints, say $v$ and $v'$, of even degree, and $e_x$ is incident on $v$ with other endpoint as $v_x$, and $e_y$ is incident on $v'$ with other endpoint as $v_y$, satisfying $v_x \ne v_y$.} We show that this case is possible, but there exists no graph $H$ such that $G=L(H)$. $v_x$, $v_y\in ov(G)$ as $e_x$ and $e_y$ are critical with $v, v'\in ev(G)$. Since $L(G)$ has at least four odd degree vertices (Observation~\ref{obs:4-ov}), $G$ must have a critical edge other than $e_x$ and $e_y$. The vertices $v$ and $v'$ already have a non-critical edge $e_a$ incident on them and critical edge $e_x$ and $e_y$ respectively incident on them. If there exists an edge incident on any of $v$ or $v'$ other than $e_a$, $e_x$ and $e_y$, then by Lemma~\ref{lemma:always-critical}, there will yet another critical edge in $L(G)$ other than $ax$ and $ay$ --- a contradiction. Therefore, $N_G(v) = \{v_x, v'\}$ and $N_G(v') = \{v, v_y\}$.
    
    Now, $v_x$ (resp. $v_y$) does not have any neighbour $w\in ov(G)$ as it already has a critical edge $e_x$ (resp. $e_y$) incident on it, and the edge $wv_x$ (resp. $wv_y)$ will be a non-critical; this will result in an yet another critical edge in $L(G)$ distinct from $ax$ and $ay$ by Lemma~\ref{obs:critical-and-non-critical} --- a contradiction. So, all neighbours of $v_x$ or $v_y$ are in $ev(G)$. Since $N_G(v) \cup N_G(v') = \{v_x, v_y, v, v\}$, there must be at least one more neighbour of $v_x$ or $v_y$ other than $v$ and $v'$ respectively as the graph is connected and is not isomorphic to a path.
    
    Without loss of generality, let $v_y$ have a neighbour other than $v'$. Now, the degree of $v_y$ is at least $3$ as it has an odd degree. Also, $N_G(v_y)\subseteq ov(G)$ as discussed. Thus, $v_y$ has at least two neighbours $v_1, v_2\in ev(G)$ other than $v'$. The vertices $v_x$, $v'$, $v_1$ and $v_2$ induce a claw in $G$ as otherwise there will be yet another critical edge in $L(G)$ distinct from $ax$ and $ay$ by Lemma~\ref{lem:claw-critical} --- a contradiction. Therefore, there is no graph $H$ such that $G=L(H)$~\cite{beineke1970characterization}, i.e. $G$ is not a line graph of any graph.
    
    \subparagraph*{Case V: $e_a$ has both endpoints, say $v$ and $v'$, of even degree, and $e_x$ is incident on $v$ with the other endpoint as $v_{xy}$, and $e_y$ is incident on $v'$ with the other endpoint as $v_{xy}$} Then $v_{xy} \in ov(G)$ as $e_x$ is critical with $v\in ev(G)$. Now, $v_{xy}$ must have at least one neighbour other than the $v$ and $v'$, as it has an odd degree. Let $u \in N_G(v_{xy}) \setminus \{v, v'\}$ be such a neighbour. Now, if $u\in ov(G)$, then $uv_{xy}$ is a non-critical edge and $v_{xy}v$ is a critical edge, both incident on $v$. So, by Observation~\ref{obs:critical-and-non-critical} there will be yet another critical edge in $L(G)$ other than $ax$ and $ay$ --- a contradiction. So, $u\in ev(G)$. Since $u$ has even degree, it must have at least one neighbour other than $v_{xy}$. Let $u' \in N_G(u)\setminus \{v_{xy}\}$ be such a neighbour. If $u'\in ev(G)$, then $u$ has both a non-critical edge $uv_{xy}$ and a critical edge $uu'$ incident on it. So, by Observation~\ref{obs:critical-and-non-critical} there will be yet another critical edge in $L(G)$ other than $ax$ and $ay$ --- a contradiction. So, $u'\in ov(G)$.
    
    Assume that $G \ne L(EG_0)$. Therefore, $G$ must have at least one edge outside than $M = \{uu',uv_{xy},e_x,e_y,e_a \}$, as these edges induce $L(EG_0)$. 
    
    Note that $v$ (resp. $v'$) does not have a neighbour other than $v_{xy}$ and $v'$ (resp. $v$). Otherwise, since there exists a critical and non-critical edge incident on $v$ (resp. $v'$), by Lemma~\ref{lemma:always-critical}, this will lead to yet another critical edge in $L(G)$ other than $ax$ and $ay$ --- a contradiction. Thus, there must exist an edge in $E(G) \setminus M$ which is incident on one of $v_{xy}$, $u$ or $u'$ as the graph is connected. In either case, there will be a claw induced with $v_{xy}$, $u$ or $u'$ as the degree $3$ vertex respectively or there will be yet another critical edge in $L(G)$ other than $ax$ and $ay$ by Lemma~\ref{lem:claw-critical} which would be a contradiction. Therefore, there is no graph $H$ such that $G=L(H)$~\cite{beineke1970characterization}, i.e. $G$ is not a line graph of any graph.

    The cases I, II, III, IV, and V are exhaustive. So, $G$ always has an induced claw except when $G=L(EG_0)$. Hence, except when $G = L(EG_0)$ i.e. $H=EG_0$, there is no graph $H$ such that $G=L(H)$.
\end{proof}

\begin{lemma}\label{lem:matching-edges}
    Let $G$ be a graph such that $L(G)$ does not have an Euler path, but $L^2(G)$ does. If both critical edges of $L(G)$ share no vertex as an endpoint, then there is no graph $H$ such that $L(H) = G$.
\end{lemma}

\begin{proof}
    Let the two critical edges in $L(G)$ be $ax$ and $by$ where $a, b\in ov(L(G))$ and $x,y\in ev(L(G))$ (Figure~\ref{fig:matching-edges}). Observe that $e_a$ and 
    $e_x$ will have a common endpoint in $G$, and so will $e_b$ and $e_y$. Moreover, $e_a$ and $e_b$ are critical edges of $G$, while $e_x$ and $e_y$ are non-critical edges. The following cases arise.
    
    \begin{figure}[ht!]
        \centering
        \includegraphics[width=0.6\textwidth]{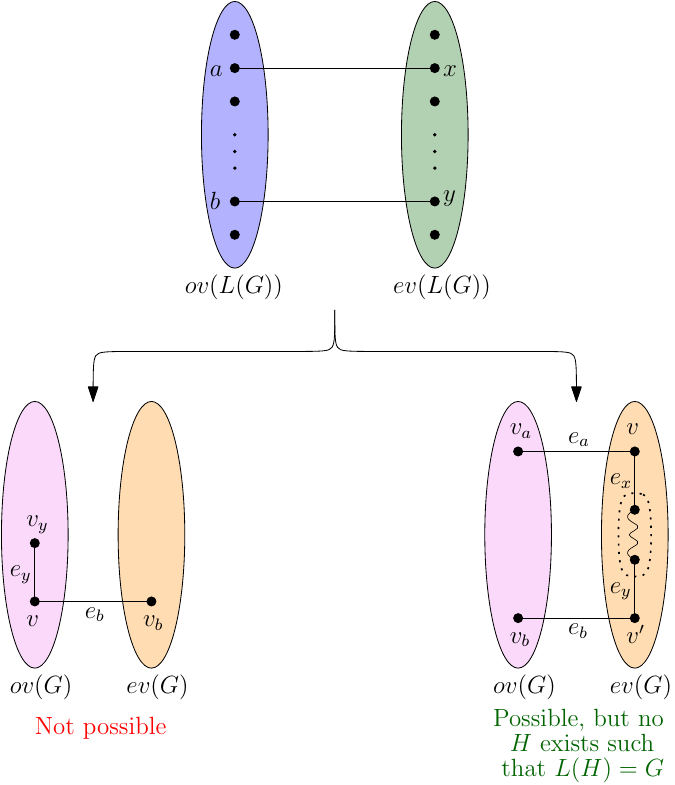}
        \caption{Cases of $G$ when both critical edges of $L(G)$ do not have a common endpoint.}\label{fig:matching-edges}
    \end{figure}
    
    \subparagraph*{Case I: Let $e_b$ and $e_y$ be incident on a common vertex with an odd degree, say $v\in ov(G)$.} Let the other endpoints of $e_b$ and $e_y$ be $v_b$ and $v_y$ respectively. Then, $v$ must have a neighbour other $v_b$ and $v_y$ as $v$ has an odd degree. Now, $v$ has a critical edge $e_b$ and a non-critical edge $e_y$ incident on it. So, by Lemma~\ref{lemma:always-critical}, any additional neighbour of $v$ will lead to yet another critical edge in $L(G)$ distinct from $ax$ and $ay$ --- a contradiction. So, this case is not possible. Similarly, the case where $e_a$ and $e_x$ have a common endpoint with an odd degree, is not possible.
    
    \subparagraph*{Case II: $e_x$ and $e_a$ are incident on $v\in ev(G)$ and $e_y$ and $e_b$ are incident on $v'\in ev(G)$.} We show that this case is possible, but there exists no graph $H$ such that $G=L(H)$. Let the other endpoint of $e_a$ be $v_a$ and that of $e_b$ be $v_b$. Vertices $v_a$, $v_b\in ov(G)$ as $e_a$ and $e_b$ are critical edges with $v,v'\in ev(G)$.

    Now, $v_a$ (resp. $v_b$) does not have any neighbour $w\in ov(G)$ as it already has a critical edge $e_a$ (resp. $e_b$) incident on it, and the edge $wv_a$ (resp. $wv_b$) will be a non-critical; this will result in an yet another critical edge in $L(G)$ distinct from $ax$ and $ay$ by Lemma~\ref{obs:critical-and-non-critical} --- a contradiction. Therefore, we have $N_G(v_a)\subseteq ev(G)$ and $N_G(v_b)\subseteq ev(G)$.

    Note that $|ov(L(G)| \ge 4$, by Observation~\ref{obs:4-ov}. So, $G$ must have at least $4$ critical edges, i.e., there exists some critical edge $e' \in E(G)$ other than $e_a$ and $uv_a$. We now show that there exists an induced claw in $G$. We look into the following cases.
    
    \textit{Case (a): $d_G(v_a) > 1$ or $d_G(v_b)>1$.} There exist edges $wv_a$ and $w'v_a$ (resp. $wv_b$ and $w'v_b$) other than $e_a$ (resp. $e_b$) incident on $v_a$ (resp. $v_b$) as it has an odd degree. Then, $w,w'\in N_G(v_a) \subseteq ev(G)$ (resp. $w,w'\in N_G(v_b) \subseteq ev(G)$). The vertices $v_a$, $v$ (resp. $v_b$ and $v'$), $w$ and $w'$ induce a claw in $G$ as otherwise there will be yet another critical edge in $L(G)$ distinct from $ax$ and $ay$, by Lemma~\ref{lem:claw-critical} ---  a contradiction.

    \textit{Case (b): $d_G(v_a) = 1$ and $d_G(v_b) = 1$.} Consider a path from $v_a$ (resp. $v_b$) to either endpoint of $e'$; such a path must contain $e_a$ (resp. $e_b$). Therefore, the critical edge $e'$ must be incident on a vertex $u$ in the connected component $C$ of $G\setminus\{e_a, e_b\}$ which includes $v$ and $v'$. Further without loss of generality, assume $u\in V(C)$ is the nearest vertex to $v$ in $C$ with a critical edge incident on it. Let $e' = uw$. Now, since $u$ is connected to $v$ in $C$ and no other vertex in the shortest path from $v$ to $u$ in $C$ has a critical edge incident on it (due to the choice of $u$), there exists a non-critical edge $w'u$, appearing in the shortest path from $v$ to $u$ in $C$ (or the non critical edge $e_x$ if $u = v$ or the non-critical edge $e_y$ if $u=v'$). So, $u$ has a critical edge and a non-critical edge incident on it. So, by Observation~\ref{obs:critical-and-non-critical}, there will be yet another critical edge in $L(G)$ other than $ax$ and $ay$ --- a contradiction. So, this case is not possible.

    The cases I and II are exhaustive. So, $G$ must contain an induced claw. Hence, there is no graph $H$ such that $G=L(H)$~\cite{beineke1970characterization}, i.e. $G$ is not a line graph of any graph.
\end{proof}
For a graph $H$, by Lemma~\ref{lem:share-odd}, Lemma~\ref{lem:share-even} and Lemma~\ref{lem:matching-edges}, we get the following theorem.
\begin{theorem}\label{thm:euler-EG1}
If $L^2(H)$ does not have an Euler path but $L^3(H)$ does, then $H = EG_0$.
\end{theorem} 

We now analyze the cases where both $G$ and $L(G)$ have an Euler path. We define a \emph{trailing path} as follows.

\begin{definition}[Trailing path]\label{def:trailing-path}
    A trailing path of a graph $G$ is a simple path which starts 
    at a vertex with degree $1$ in $G$, has each internal vertex of degree $2$ in $G$, and finally ends at a vertex having any even degree which is at least $4$ in $G$,
    i.e., a path of the form $v_1,v_2,v_3,...,v_k$ where $d_G(v_1)=1$, $d_G(v_i)=2$ for all $i$, $2\leq i \leq k-1$ and $d_G(v_k) \ge 4$, such that $d_G(v_k)$ is even, for some $k \ge 2$. The length of such a trailing path is $(k - 1)$.
\end{definition}

 Firstly, note that if $v_1, v_2, \ldots, v_k$ is a trailing path, then $v_iv_{i+1}$ is a critical edge for $i = 1$ and non-critical for $i \in \{2, 3, \ldots, k-1\}$. Secondly, observe that if $G$ has an Euler path then there are exactly two vertices of odd degree in $G$ and therefore at most two degree $1$ vertices in $G$. Therefore, there are at most two distinct trailing paths in $G$.

\begin{theorem}\label{thm:trail-path}
If $G$ has an Euler path, then $L(G)$ also has an Euler path if and only if one of the following holds true.
\begin{enumerate}
\item $G$ has exactly two trailing paths. Moreover, the largest value of $k$ such that $L^k(G)$ has an Euler path is equal to the minimum of the lengths of the two trailing paths in $G$, say $l$. Also, all of the graphs $L^i(G),1\leq i\leq l$ have an Euler path.
\item The two odd degree vertices in $G$ are adjacent, one has degree equal to $1$ and the other has degree equal to $3$. Here, there 
exists no $k\geq 2$ such that $L^k(G)$ has an Euler path.
\end{enumerate}
\end{theorem}

\begin{proof}
Since $G$ has an Euler path, there are exactly two odd degree vertices in $G$. Let the two odd degree vertices be $u$ and $v$. 
Now, since $L(G)$ has an Euler path, there are exactly two critical edges in $G$ by Observation \ref{obs:odd-critical}. Now, there will be two cases:

\subparagraph{Case I: $u$ and $v$ are not adjacent.} $u$ and $v$ must have only even degree vertices as neighbours, and hence all edges incident on $u$ and $v$ are critical edges. Since $G$ has exactly two critical edges, $u$ and $v$ must have exactly one neighbour each, i.e. $d_G(u)=d_G(v)=1$. So, there are exactly two trailing paths in $G$ (Since $G$ is a connected graph not isomorphic to a path). Let their lengths be $l$ and $l'$ respectively with $l\leq l'$. We prove the following claims, completing the analysis for this case.

\begin{claim}\label{clm:0-to-l-1}
    For $k \in \{0, 1, \ldots, l - 1\}$, $L^k(G)$ has an Euler path, and exactly two trailing paths, of length $l - k$ and $l' - k$.
\end{claim}
\begin{claimproof}
    We proceed via induction. For $k = 0$, $L^0(G) = G$ has an Euler path and two trailing paths of length $l$ and $l'$, by assumption. Assume for some $k \in \{0, \ldots, l - 2\}$, $L^k(G)$ has an Euler path and two trailing paths of lengths $p = l - k \ge 2$ and $p' = l' - k \ge 2$. Let these two trailing paths of $L^k(G)$ be $a,v_1,v_2,\ldots,v_p$ and $b,v_1',v_2',\ldots,v_{p'}'$ such that $d_{L^k(G)}(a)=d_{L^k(G)}(b)=1$, $d_{L^k(G)}(v_i)=2$ for $i\in\{1,2,\ldots,p-1\}$, $d_{L^k(G)}(v_i')=2$ for $ i\in\{1,2,\ldots,p'-1\}$ and $d_{L^k(G)}(v_p),d_{L^k(G)}(v'_{p'}) \geq 4$ is even. Since $L^k(G)$ has an Euler path, $ov(L^k(G)) = \{a, b\}$; this implies that $av_1, bv'_1 \in E(L^k(G))$ are the only critical edges of $L^k(G)$.  Let $w, x, y, z \in V(L^{k+1}(G))$ correspond to the edges $av_1, bv_1', v_{p-1}v_p, v_{p'-1}'v_{p'}' \in E(L^k(G))$ respectively. $w$ and $x$ are the only odd degree vertices in $L^{k+1}(G)$ as they correspond to the critical edges $av_1 \in E(L^k(G))$ and $bv_1' \in E(L^k(G))$; all other vertices of $L^{k+1}(G)$ will be of even degree as remaining edges of $L^k(G)$ are non-critical. Therefore $L^{k+1}(G)$ has an Euler path. Moreover, $L^{k+1}(G)$ exactly two trailing paths, starting at vertices $w$ and $x$ respectively (each having degree $1$) and ending at vertices $y$ and $z$ respectively (each having degree at least $4$). So, the length of the two trailing paths in $L^{k+1}(G)$ is $p-1 = l - k - 1$ and $p'-1 = l - k - 1$. This proves the claim.
\end{claimproof}

\begin{claim}
    $L^l(G)$ has an Euler path and for all $k \ge l + 1$, $L^k(G)$ does not have an Euler path.
\end{claim}
\begin{claimproof}
    From Claim~\ref{clm:0-to-l-1}, $L^{l-1}(G)$ has an Euler path and exactly two trailing paths, of lengths $1$ and $p' = l' - l - 1$. Let $a, v_1$ and $b, v'_1, v'_2, \ldots, v'_{p'}$ be the trailing paths of $L^{l - 1}(G)$. Therefore, $d_{L^{l - 1}(G)}(a) = d_{L^{l-1}(G)}(b) = 1$, $d_{L^{l - 1}(G)}(v_1),  d_{L^{l-1}(G)}(v'_{p'}) \ge 4$ is even, and $d_{L^{l-1}(G)}(v'_{i}) =2$ for $i \in \{1, 2, \ldots, p' - 1\}$. Since $L^{l - 1}(G)$ has an Euler path, $ov(L^{l-1}(G)) = \{a, b\}$. Moreover, since $v_1$ is the only neighbour of $a$ and $v'_1$ is the only neighbour of $b$ in $L^{l-1}(G)$, $av_1, bv'_1$ are the only critical edges of $L^{l-1}(G)$. Let $x$ and $y$ be vertices in $L^l(G)$ corresponding to edges $av_1, bv'_1 \in L^{l-1}(G)$ respectively. Therefore, $L^l(G)$ has an Euler path, and $ov(L^l(G)) = \{x, y\}$. Now, either $x$ and $y$ are adjacent vertices in $L^l(G)$ or they are not. We show that in either case $L^{l + 1}(G)$ has no Euler path.

    If $x$ and $y$ can be adjacent if and only if $p' = 1$ and $v'_1 = v_1$ (as $av_1$ and $bv'_1 = bv_1$ would be both incident on $v_1$ in $L^{l-1}(G)$, implying $xy \in E(L^l(G))$). By Proposition~\ref{Deguv}, $d_{L^l(G)}(x) = d_{L^l(G)}(y) = 1 + d_{L^{l-1}(G)}(v_1) - 2 \ge 3$. Therefore both $x$ and $y$ have at least two neighbours each, outside $\{x,y\} = ov(L^l(G))$. Hence, there are at least four critical edges in $L^l(G)$, implying that $|ov(L^{l+1}(G))| \ge 4$, i.e. $L^{l+1}(G)$ has no Euler path.

    If $x$ and $y$ are not adjacent, then by Proposition~\ref{Deguv}, $d_{L^l(G)}(x) = d_{L^{l-1}(G)}(a) + d_{L^{l-1}(G)}(v_1) - 2 \ge 3$. Moreover, all neighbours of $x$ are outside $ov(L^l(G)) = \{x,y\}$. Therefore $L^{l}(G)$ has at least three critical edges in $L^l(G)$, implying that $|ov(L^{l+1}(G))| \ge 3$, i.e. $L^{l+1}(G)$ has no Euler path.

    Finally, we show that $L^k(G)$ for $k \ge l + 2$ has no Euler path. For the sake of contradiction, assume there exists a smallest $k$ such that $L^k(G)$ has an Euler path, for some $k \ge l +2 \ge 3$. By the minimality of $k$, $L^{k-1}(G)$ has no Euler path. By Theorem~\ref{thm:euler-EG1}, this implies $L^{k-3}(G) = EG_0$. Since $EG_0$ is the line graph of no graph, we cannot have $k > 0$. Therefore $k=0$, implying $G = EG_0$. However $EG_0$ has no Euler path --- a contradiction to our assumption that $G$ has an Euler path. Therefore, $L^k(G)$ for $k \ge l + 2$ has no Euler path.
\end{claimproof}


\subparagraph{Case II: $u$ and $v$ are adjacent.} Then at least one of $u$ and $v$ must have degree at least $3$ since $G$ has at least three vertices. Without loss of generality, let $d_G(u) \geq 3$. Now, there are at least two other neighbours $u_1$, $u_2$ of $u$ other than $v$ and therefore $u_1$ and $u_2$ have an even degree. Therefore, there are at least two critical edges $uu_1$, $uu_2$ of $G$ incident on $u$. Since $G$ has exactly two critical edges, $uu_1$ and $uu_2$ are the only critical edges of $G$. Observe that, $u$ cannot have any neighbour $w$ other than $u_1$, $u_2$ and $v$ (as $uw$ would be a critical edge), and similarly $v$ cannot have any neighbour $w'$ other than $u$ (as $vw'$ would be a critical edge). 
Therefore, $d_G(v)=1$ and $d_G(u) = 3$. In this case, $uu_1$ and $uu_2$ correspond to the two odd degree vertices, say $x$ and $y$ respectively, in $L(G)$. Note that $xy \in E(L(G))$ as $uu_1 \in E(G)$ and $uu_2 \in E(G)$ are both incident on $u$. Moreover, by Proposition~\ref{Deguv}, $d_{L(G)}(x) = d_G(u) + d_G(u_1) - 2 \ge 3$ and $d_{L(G)}(y) = d_G(u) + d_G(u_2) - 2 \ge 3$. As $x$ and $y$ have at least two neighbours in $L(G)$ outside $\{x, y\} = ov(L(G))$, $L(G)$ has at least $4$ critical edges. This implies that $L^2(G)$ cannot have an Euler path. Suppose there is a $k> 2$ such that $L^k(G)$ has an Euler path, then Theorem~\ref{thm:euler-EG1} rules out the existence of $L(G)$, which is a contradiction. Thus, there does not exist $k\geq 2$ such that $L^k(G)$ has an Euler path. 
\end{proof}


\subsection{Efficient algorithm for finding all higher order line graphs having an Euler path}

We now start the discussion on the algorithmic aspect. Consider the following problem. 

\defproblem{Line Graph Euler Path Indices (LGEPI)}{A simple connected graph $G$.}{Output the set of all integers $k$ such that $L^k(G)$ has an Euler path.}

We design a polynomial time algorithm for \LGEPI{} and as a consequence, we also derive a polynomial bound on the largest $k$ such that $L^k(G)$ has an Euler path. Before we design the algorithm for \LGEPI{}, we make an observation regarding the graph $EG_0$ (Figure~\ref{fig:EG_0}).

\begin{observation}\label{obs:eg0}
    $L^k(EG_0)$ has an Euler path if and only if $k \in \{1, 3\}$.
\end{observation}

\begin{theorem}\label{thm:alg-LGEPD}
    There exists an algorithm solving \LGEPI{} on a connected graph $G$ with $n$ vertices and $m$ edges in time $\mathcal{O}(n^2m)$.
\end{theorem}
\begin{proof}
    Consider the following algorithm. For an input graph $G$, we first check if $G = EG_0$, and in that case we output $\{1, 3\}$ (by Observation~\ref{obs:eg0}). Otherwise, we compute the line graphs $L(G)$, $L^2(G)$, and the set $ov(L^3(G))$. Note that $\Delta(G), \Delta(L(G)), \Delta(L^2(G)), \Delta(L^3(G)) = \mathcal{O}(n)$. Hence, $|E(L^2(G))| = |V(L^3(G))| \le \mathcal{O}(n) \cdot |E(L(G))| \le \mathcal{O}(n^2) \cdot |E(L^2(G))| = \mathcal{O}(n^2m)$. We can compute these in time $\mathcal{O}(n^2m)$. Moreover, we can check whether $G$, $L(G)$, $L^2(G)$ and $L^3(G)$ have Euler paths by just checking the number of odd degree vertices. We now consider $L^k(G)$ for $k \ge 4$. We analyze the following cases.

    \begin{itemize}
        \item If $L^2(G)$ has no Euler path, then $L^3(G)$ cannot have an Euler path as $G \ne EG_0$ (by Theorem~\ref{thm:euler-EG1}). Moreover for $k \ge 3$, if $L^k(G)$ has no Euler path, then $L^{k+1}(G)$ cannot have an Euler path, as Theorem~\ref{thm:euler-EG1} would imply $L^{k-2}(G) = L(L^{k-2}(G)) = EG_0$, a contradiction as $EG_0$ is not a line graph of any graph (since $EG_0$ has an induced claw). Therefore, none of $L^2(G), L^3(G), L^4(G), \ldots$ have Euler paths. 
        \item Otherwise if $L^2(G)$ has an Euler path, then either $L^3(G)$ has an Euler path or it does not. We work out these two cases as follows.
        \begin{itemize}
            \item If $L^3(G)$ has no Euler path, then we show that none of $L^4(G), L^5(G), \ldots$ have an Euler path. For the sake of contradiction, assume that for some $k \ge 3$, $L^{k}(G)$ has no Euler path but $L^{k+1}(G)$ has an Euler path. By Theorem~\ref{thm:euler-EG1}, $L^{k-2} = L(L^{k-3}(G))=EG_0$. But this is a contradiction as $EG_0$ is not a line graph of any graph.
            \item If $L^3(G)$ has an Euler path, then by Theorem~\ref{thm:trail-path}, we have the following cases.
            \begin{itemize}
                \item If $L^2(G)$ has two trailing paths, then $L^{i}(G)$ has an Euler path when $4\leq i \leq 2+\ell$ but not when $i\geq 2+\ell+1$, where $\ell$ is the length of the smallest trailing path in $L^2(G)$. Note that $\ell$ can be computed in $\mathcal{O}(|E(L^2(G))|) = \mathcal{O}(n^2m)$ time. 
                \item Otherwise, $L^2(G)$ must have two adjacent odd degree vertices: one vertex of degree $1$, and another of degree $3$. Then, none of $L^4(G), L^5(G), \ldots$ have an Euler path.
            \end{itemize}
        \end{itemize}
    \end{itemize}
Therefore, \LGEPI{} admits an algorithm with running time $\mathcal O(n^2m)$. We provide the pseudocode of the algorithm for completeness (see Algorithm~\ref{alg:LGEPI}).

    \begin{algorithm}[H]
        \caption{ \hfill \textbf{Input:} Graph $G$ \hfill \textbf{Output:} All $k$ such that $L^k(G)$ has an Euler path}\label{alg:LGEPI}
        \begin{algorithmic}[1]
            \If{$G = EG_0$}
                \State{\Return $\{1, 3\}$}
            \EndIf
            \State{Compute $L(G)$, $L^2(G)$ and $ov(L^3(G))$} \Comment{In time $\mathcal{O}(n^2m)$}
            \State{$S \gets \emptyset$} \Comment{We would compute our final output in $S$}
            \For {$i \in \{0, 1, 2, 3\}$}
                \If{$|ov(L^i(G))| = 2$} \Comment{$L^i(G)$ has an Euler path}
                    \State{$S \gets S \cup \{i\}$}
                \EndIf
            \EndFor
            \If{$|ov(L^2(G))| \ne 2$} \Comment{$L^2(G)$ has no Euler path}
                \State{\Return $S$}
            \Else \Comment{$L^2(G)$ has an Euler path}
                \If{$|ov(L^3(G))| \ne 2$} \Comment{$L^3(G)$ has no Euler path}
                    \State{\Return $S$}
                \Else \Comment{$L^3(G)$ has an Euler path}
                    \If{$L^2(G)$ has two trailing paths of length at least $1$}
                        \State{$p, p' \gets$ lengths of the trailing paths of $L^2(G)$} \Comment{$\mathcal{O}(|E(L^2(G))|)$ time}
                        \State{\Return $S \cup \{4, 5, \ldots, \min(p,p') + 2\}$}
                    \Else \Comment{$L^2(G)$ has two adjacent vertices of degrees $1$ and $3$ respectively}
                        \State{\Return $S$}
                    \EndIf
                \EndIf
            \EndIf
        \end{algorithmic}
    \end{algorithm}

\end{proof}

Note that the largest value of $k$ such that $L^k(G)$ has an Euler path is at most $2 + |V(L^2(G))| = \mathcal{O}(|V(G)||E(G)|)$.

\begin{corollary}
    For every graph $G$ of $n$ vertices and $m$ edges, there exists an integer $k_0 = \mathcal{O}(nm)$ such that $L^k(G)$ does not have an Euler path for all $k \ge k_0$.
\end{corollary}

\section{Maximum Degree Growth on Higher Order Line Graphs}

In this section, we study the growth of the maximum degree in $L^k(G)$, $k \ge 0$, for a graph $G$. When $G$ is a path, a cycle or a claw, the maximum degree becomes of $L^k(G)$ stays $0$, $1$ or $2$ for $k \ge 1$. Hence we restrict our focus to other simple connected graphs, i.e. graphs $G$ that satisfy $G\in \mathcal G=\{G \mid G\text{ is a simple connected graph, }\Delta(G)\geq3 \text{ and }|E(G)|\geq 4\}$; $\mathcal G$ is called the set of \emph{prolific graphs}~\cite{caro2022index}.

We start by showing that the works of Hartke and Higgins~\cite{stephen1999growth} imply that for any prolific graph $G$, there exists a constant rational number $\dgc(G)$ and an integer $k_0$ such that for all $k \ge k_0$, $\Delta(L^k(G)) = \dgc(G) \cdot 2^{k-4} + 2$. We show that $\{\dgc(G) \mid G \in \mathcal G\}$ has first, second, third, fourth, and fifth minimums, namely, $c_1 = 3$, $c_2 = 4$, $c_3 = 5.5$, $c_4 = 6$ and $c_5=7$; the third minimum stands out surprisingly from the other four. We use structural properties of subgraphs and their corresponding images under line graphs as a tool to compare the parameter $\dgc$ of two distinct graphs. Moreover, for $i \in \{1, 2, 3, 4\}$, we provide a complete characterisation of $\mathcal G_i = \{\dgc(G) = c_i \mid G \in \mathcal G \}$. Apart from this, we show that the set $\{\dgc(G) \mid G \in \mathcal G, 7 < \dgc(G) < 8\}$ is countably infinite by providing a family of graphs attaining distinct values of $\dgc(G)$ in this range.

\subsection{Definitions and tools}\label{subsec:toolsmaxdeg}

We start by reiterating the definition of the \textit{Maximum Degree Growth Property Index} ($\MDGPI$) from the works of Hartke and Higgins~\cite{stephen1999growth}. Their work proves that for every prolific graph, $\MDGPI$ exists.

\begin{definition}[Maximum Degree Growth Property Index ($\MDGPI$)]\label{def:MDGPI}
The $\MDGPI$ of a graph $G$, denoted $\MDGPI(G)$, is the smallest integer $K$ such that for all $k \ge K+1$,
\[\Delta(L^k(G)) = 2\cdot\Delta(L^{k-1}(G)) - 2\]
\end{definition}

\begin{observation}\label{Formula}
    If $\MDGPI(G)\leq K$ for a graph $G$, then for every $k \ge K$, 
    \[\Delta(L^k(G))=(\Delta(L^K(G))-2)\cdot 2^{k-K}+2\]
\end{observation}

\begin{proof}
    We proceed via induction. This holds trivially for $k=K$. Now, we assume this to hold true for some $k \ge K$. Therefore, $\Delta(L^k(G))=(\Delta(L^K(G))-2)\cdot 2^{k-K}+2$. From Definition~\ref{def:MDGPI} we have,
    \begin{align*}
        \Delta(L^{k+1}(G)) &= 2 \cdot \Delta (L^k(G)) - 2 = 2 \cdot ((\Delta(L^K(G))-2)\cdot 2^{k-K}+2) - 2 \\
        &= (\Delta(L^K(G))-2)\cdot 2^{(k+1)-K}+2
    \end{align*}
    This completes the inductive proof.
\end{proof}

\begin{definition}\label{def:delta-cycle}
    For a graph $G$, $\Delta(G)$ is the highest degree of a graph $G$. A vertex having degree $\Delta(G)$ is called a $\Delta$-vertex, and a vertex having degree $\Delta(G)-1$ is called a $(\Delta - 1)$-vertex. A $\Delta$-cycle of $G$ is a cycle where each node has degree $\Delta(G)$ in $G$. A $\Delta$-triangle is a $\Delta$-cycle of length $3$.
\end{definition}

Their work deduce further propoerties of $\MDGPI$.

\begin{theorem} \label{MDGPIDeltaCycle}
    If for a graph $G$, the graph $L^k(G)$ contains a $\Delta$-cycle, then $\MDGPI(G)\leq k$.
\end{theorem}

Putting these together, we get the following observation.

\begin{observation}\label{DirectFormula}
    Let $L^k(G)$ contain a $\Delta$-triangle with $\Delta(L^k(G))=x$. Then, $\Delta(L^r(G))=(x-2)\cdot2^{r-k}+2$ for all $r \geq k$.
\end{observation}

\begin{proof}
    As $L^k(G)$ has a $\Delta$-triangle, $\MDGPI(G)\leq k$ due to Theorem~\ref{MDGPIDeltaCycle}. Hence, for any $r \geq k$, $\Delta(L^r(G))=(x-2)\cdot2^{r-k}+2$ by Observation~\ref{Formula}.
\end{proof}

Using the above observation, we get the exact value of $\Delta(L^r(G))$ if $L^k(G)$ has a $\Delta$-triangle for $k\leq r$. We now provide the following result that gives a lower bound on $\Delta(L^r(G))$ if $L^k(G)$ has a triangle whose vertices satisfy a lower bound on their degree for $k\leq r$.

\begin{lemma}\label{lem:DeltaOrMoreTriangle}
    Let $L^k(G)$ contain a triangle with vertices $a,b,c\in V(L^k(G))$. Let for some $d\in\mathbb N$, $d_{L^k(G)}(a),d_{L^k(G)}(b),d_{L^k(G)}(c)\geq d$. Then, for all $r\geq k$, $L^r(G)$ has a triangle with vertices $a',b',c'\in V(L^r(G))$ such that $d_{L^r(G)}(a'),d_{L^r(G)}(b'),d_{L^r(G)}(c')\geq (d-2)\cdot2^{r-k}+2$. Consequently, $\Delta(L^r(G))\geq(d-2)\cdot2^{r-k}+2$ for all $r \geq k$.
\end{lemma}

\begin{proof}
    We proceed via induction. For $r=k$, this holds true. Assume the claim to hold true for some for some $r\geq k$. So, $L^r(G)$ has a triangle with vertices $a',b',c'\in V(L^r(G))$ such that $d_{L^r(G)}(a'),d_{L^r(G)}(b'),d_{L^r(G)}(c')\geq (d-2)\cdot2^{r-k}+2$. For the edge $a'b'\in E(L^r(G))$ (resp. $b'c'$ and $c'a'$), the degree of vertex $v_1$ (resp. $v_2$ and $v_3$) corresponding to it in $L^{r+1}(G)$ is $d_{L^{r+1}(G)}(v_1)=d_G(a')+d_G(b')-2\geq (d-2)\cdot2^{r-k+1}+2$. Also, vertices $v_1$, $v_2$ and $v_3$ form a triangle in $L^{r+1}(G)$ because their corresponding edges in $G$ pairwise share an endpoint with each other. Hence, we are done. 
\end{proof}

Recall that Hartke and Higgins~\cite{stephen1999growth} show that $\MDGPI$ exists and is finite of any connected graph $G$. By Observation~\ref{Formula}, we get that there exists a constant $k_0 = \MDGPI(G) + 1$, such that for all $k \ge k_0$, $\Delta(L^k(G)) = (\Delta(L ^ {k_0}(G)) - 2) \cdot 2^{k-k_0} + 2$. Therefore, for every graph $G$, there exists a constant $c$, such that asymptotic growth of $
\Delta (L^k(G))$ is given by $c \cdot 2^{k} + 2$, in particular where $c = {(\Delta(L ^ {k_0}(G)) - 2)} \cdot {2^{-k_0}}$. We would denote sixteen times this constant as the \emph{degree growth constant} ($\dgc(G)$). The choice of the factor of sixteen is arbitrary from a mathematical standpoint, and is purely chosen for a better intuition, which would be evident later in the paper. We state this as a proposition. 

\begin{proposition}\label{prop:dgc}
    For every connected graph $G$ there exists a unique rational number $c'$, such that for all integers $k \ge \MDGPI(G) + 1$, we have $\Delta(L^k(G)) = c' \cdot 2^{k - 4}  + 2$. We call this constant $c'$ the degree growth constant, and denote it by $\dgc(G)$.
\end{proposition}

With the definition of $\dgc(G)$ laid out, we restrict our focus to analyzing the exact values that $\dgc(G)$ can attain for $G \in \mathcal{G}$. We start by defining two binary relations among graphs based on the maximum degrees of their higher order line graphs.

\begin{definition}
    Consider two graphs $G_1, G_2\in \mathcal G$. We say $G_1\leq_k G_2$ for some non-negative integer $k$, if $\Delta(L^r(G_1))\leq \Delta(L^r(G_2))$ for all $r\geq k$.
\end{definition}

This binary relation is related to the degree growth constant as follows.

\begin{observation}\label{obs:r1todgc}
    If for two graphs $G_1$ and $G_2$, we have $G_1\leq_k G_2$ for some non-negative integer $k$, then $\dgc(G_1)\leq \dgc(G_2)$.
\end{observation}

\begin{definition}
    Consider two graphs $G_1, G_2\in \mathcal G$. Given two non-negative integers $a,b$ with $a \leq b$, we say $G_1\leq_{a,b}G_2$ if $L^a(G_1)\subseteq L^b(G_2)$, i.e. $L^a(G_1)$ is a subgraph of $L^b(G_2)$.
\end{definition}

The binary relations $\leq_{x,y}$ and $\leq_{x}$ can be shown to be related as follows.

\begin{lemma}\label{Comp1}
    If $G_1\leq_{x,y}G_2$, then $G_1\leq_xG_2$.
\end{lemma}

\begin{proof}
    Since $G_1\leq_{x,y}G_2$, we have that $L^x(G_1)\subseteq L^y(G_2)$ and $x\geq y$. Now, by Proposition~\ref{subindsub}, we have $L^{x+i}(G_1)\subseteq L^{y+i}(G_2)$ for all $i\ge 0$. Let $k \ge x$, then $L^k(G_1)\subseteq L^{y+k-x}(G_2)$. Note that $k+y-x\leq k$. We now show that $\Delta(L(G))\geq \Delta(G)$ whenever $G$ is not a star graph. Consider a $\Delta$-vertex $v\in V(G)$. Note that all neighbours of $v$ cannot have degrees equal to $1$ as otherwise $G$ will be a star graph. Therefore, $\exists u\in N_G(v)$ such that $d_G(u)\geq 2$. Let the vertex corresponding to edge $uv$ in $L(G)$ be $v'$. Then, by Proposition~\ref{Deguv}, $\Delta(L(G))\geq d_G(v')=d_G(v)+d_G(u)-2\geq \Delta(G)$.
    
    Now, higher order line graphs cannot be isomorphic to star graphs as they are claw-free. Thus, we have that $\Delta(L^k(G_1))\leq\Delta(L^{y+k-x}(G_2))\leq\Delta(L^k(G_2))$, which in other words means that $G_1\leq_xG_2$.
\end{proof}

From Observation~\ref{obs:r1todgc} and Lemma~\ref{Comp1}, we have the following.

\begin{observation}\label{obs:r2todgc}
    If for two graphs $G_1$ and $G_2$, we have $G_1\leq_{x,y}G_2$ for some non-negative integers $a\le b$, then $\dgc(G_1)\leq\dgc(G_2)$.
\end{observation}

With the definitions and the relevant preliminary results ready we establish a few more results which would finally help us in analyzing the set $\{\dgc(G) \mid G \in \mathcal G\}$.

\begin{lemma}\label{LongDeltaPath}
    Let $G$ be a graph satisfying the following:
    \begin{enumerate}
        \item The $\Delta$-vertices of $G$ induce a path $v_1, v_2, \ldots, v_l$ in $G$ of length $l-1$.
        \item There are at least three distinct edges in $G$ that connect a $\Delta$-vertex to a $(\Delta-1)$-vertex.
    \end{enumerate}
    Then, $L^l(G)$ contains a $\Delta$-triangle and $\Delta(L^l(G))=(\Delta(G)-2)\cdot 2^{l}+1$.
\end{lemma}

\begin{proof}     We start by proving the following claim.
    
    \begin{claim} 
        For all $0\leq k\leq l-1$, the following holds true:
        \begin{itemize}
            \item $\Delta(L^k(G))=(\Delta(G)-2)\cdot2^{k}+2$,
            \item the $\Delta$-vertices of $L^k(G)$ induce a path of length $l-k-1$,
            \item there exists at least three edges in $L^k(G)$ that connect a $\Delta$-vertex to a $(\Delta-1)$-vertex.
        \end{itemize} 
    \end{claim}
    \begin{claimproof}
        We proceed via induction. For the base case of $k=0$, this is true by assumption. Let the claim hold true for some $k - 1$, where $1 \leq k\leq l-1$. Let $H = L^{k-1}(G)$, then $\Delta(H) = (\Delta(G) - 2) \cdot2^{k-1} + 2$. Moreover, the $\Delta$-vertices of $H$ induce a path of length $l-k$ and there are at least three distinct edges in $H$ that connect a $\Delta$-vertex to a $(\Delta-1)$-vertex. By Proposition~\ref{Deguv}, all $\Delta$-vertices in $L(H)$ would correspond to an edge in $H$ connecting two $\Delta$-vertices in $H$; in particular, the $\Delta$-vertices of $L(H)$ would correspond to the $l-k$ edges of the path induced by the $\Delta$-vertices in $H$. So, the $\Delta$-vertices of $L(H)$ induce a path of length $l-k-1$ in $L(H)$, each having degree $\Delta(L(H)) = 2 \cdot \Delta(H) - 2 = (\Delta(G) - 2) \cdot 2^k + 2$. 
        
        Next, observe that the $(\Delta-1)$-vertices of $L(H)$ correspond to edges in $H$ which connect $\Delta$-vertices and $(\Delta-1)$-vertices. There are at least three such edges in $H$. Let three such edges be $u_1v_1$, $u_2v_2$ and $u_3v_3$, where $d_H(u_1) = d_H(u_2) = d_H(u_3) = \Delta(H) - 1$ and $d_H(v_1) = d_H(v_2) = d_H(v_3) = \Delta(H)$. Since all $\Delta$-vertices of $H$ induce a path, there is at least one $\Delta$-vertex neighbour of $v_1$, $v_2$ and $v_3$ each; Let these be $a$, $b$ and $c$ respectively. So, the $\Delta$-vertices of $L(H)$ corresponding to the edges $av_1, bv_2$ and $cv_3$ of $H$ will be adjacent to $(\Delta-1)$-vertices of $L(H)$ corresponding to edges $u_1v_1, u_2v_2$ and $u_3v_3$ of $H$ respectively. Therefore, $L(H)$ has at least three edges connecting $\Delta$-vertices and $(\Delta-1)$-vertices.
    \end{claimproof}

    For $k = l-1$, this claim says that $L^{l-1}(G)$ has exactly one $\Delta$-vertex of degree $(\Delta-2)\cdot 2^{l-1}+2$ and it is adjacent to three distinct $(\Delta-1)$-vertices. Hence, these three edges correspond to $\Delta$-vertices in $L^l(G)$. Moreover, these $\Delta$-vertices in $L^l(G)$ have degree $\Delta(L^l(G))=2\cdot\Delta(L^{l-1}(G))-3=(\Delta(G)-2)\cdot 2^{l}+1$ and induce a $\Delta$-triangle.
\end{proof}

The next four lemmata establish lower bounds based on local structures of the graph. 

\begin{lemma}\label{CycPen}
   For every graph $G\in\mathcal G$ satisfying $C_{n-1}\subseteq G$ where $C_n$ is a cycle of $n \ge 3$ vertices, we have $\Delta(L^k(G)) \ge 6\cdot2^{k-4}+2$ for all $k\geq3$, i.e. $\dgc(G) \ge 6$.
\end{lemma}

\begin{proof} 
    Let $CP_n$ be the graph on $n$ vertices, with a cycle of length $n-1$ and one of the vertices in the cycle is adjacent to a single vertex with degree equal to $1$ (Figure~\ref{CP}).
    \begin{figure}[ht!]
    \centering
    \includegraphics[width=0.35\textwidth]{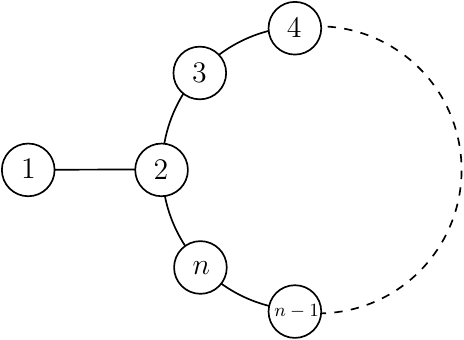}
    \caption{Cycle with a pendant, $CP_n$}
    \label{CP}
    \end{figure}
    We show that all graphs $G\in\mathcal G$ with $C_{n-1}\subseteq G$, we have $CP_n \subseteq G$, i.e. $CP_n\leq_{0,0}G$. Let $G$ contain a cycle $C$ with vertices $v_1,v_2,\ldots,v_{n-1}$ in order. Now, since $G\in\mathcal G$, $\Delta(G)\geq 3$, there must be at least one edge $e \in E(G)$ other than the edges appearing in the cycle $C$. If $e$ is between two vertices of the cycle $C$, say $v_x$ and $v_y$ with $x < y$, then vertices $v_1,v_2,\ldots,v_x,v_y,v_{y+1},\ldots,v_{n-1}$ along with the vertex $v_{x+1}$ form a cycle with pendant, which is isomorphic to $CP_{n + 1 + x - y}$; this implies $CP_{n+1+x-y} \leq_{0,0} G$. Otherwise, if $e$ connects a vertex in $C$ with a vertex $u$ outside $C$, then vertices $v_1,v_2,\ldots,v_{n-1},u$ form a cycle with a pendant, which is isomorphic to $CP_{n}$. So, $CP_n\leq_{0,0} G$. Hence in either case, there exists $n'$, such that $CP_{n'}\leq_0 G$, by Lemma~\ref{Comp1}; this implies $\Delta(L^k(G)) \ge \Delta(L^k(CP_{n'}))$. 

    Now, the graph $L(CP_{n'})$ has two $\Delta$-vertices of degree $3$ which are adjacent to each other; i.e. its $\Delta$-vertices induce a path with $2$ vertices. Moreover, $L(CP_{n'})$ also contains at least three edges connecting a $\Delta$-vertex to a $(\Delta-1)$-vertex of $L(CP_{n'})$. So, $L^3(CP_{n'})$ contains a $\Delta$-triangle with $\Delta(L^3(CP_{n'}))= (\Delta(L(G)) - 2) \cdot 2^{3 - 1} + 1 = 5$ by Lemma~\ref{LongDeltaPath}. Hence, $\Delta(L^k(CP_{n'}))=3\cdot2^{k-3}+2=6\cdot2^{k-4}+2$ for $k\geq3$ by Observation~\ref{DirectFormula}. Therefore $\Delta(L^k(G)) \ge \Delta(L^k(CP_{n'})) \ge 6 \cdot 2^{k - 4} + 2$, for all $k \ge 3$, i.e. $\dgc(G) \ge 6$.\
\end{proof}

\begin{lemma}\label{Deg4Above}
    $G\in\mathcal G$ be a graph with $\Delta(G)\geq4$, then we have $\Delta(L^k(G)) \ge 8\cdot2^{k-4}+2$ for all $k\geq1$, i.e. $\dgc(G) \ge 8$.
\end{lemma}
\begin{proof}
    Let $v$ be a $\Delta$-vertex in $G$. By assumption, $v$ has at least four distinct neighbours. Let $a,b,c,d$ be four distinct neighbours of $v$ in $G$. Let $x, y, z \in V(L(G))$ be vertices corresponding to the edges $va, vb, vc \in E(G)$ respectively. Note that $d_{L(G)}(x) = d_G(v) + d_G(a) - 2 \ge 4 + 1 - 2 = 3$ (from Proposition~\ref{Deguv}). Similarly $d_{L(G)}(y), d_{L(G)}(z) \ge 3$. Moreover, $xy, yz, za \in E(L(G))$ as edges $va, vb, vc$ are incident on $v$. Therefore, by Lemma~\ref{lem:DeltaOrMoreTriangle}, for all $k \ge 2$, $\Delta(L^{k}(G)) = (3 - 2) \cdot 2 ^{k-1} + 2 = 8 \cdot 2^{k-4} + 2$.
\end{proof}

\begin{lemma}\label{Deg3to3deg2}
   In a graph $G \in \mathcal G$ if there exists $v,v_1,v_2,v_3\in V(G)$ such that $d_G(v)\geq 3$, $d_G(v_1),d_G(v_2),d_G(v_3)\geq 2$ and $vv_1,vv_2,vv_3\in E(G)$, then $\Delta(L^k(G))=8\cdot 2^{k-4}+2$, for all $k\geq 1$, i.e. $\dgc(G) \ge 8$.
\end{lemma}

\begin{proof}
    Let the vertices corresponding to edges $vv_1$, $vv_2$ and $vv_3$ in $L(G)$ be $u_1$, $u_2$ and $u_3$ respectively. So, $d_G(u_i) = d_G(v)+d_G(v_i)-2 \ge 3$ for all $i\in\{1,2,3\}$ by Proposition~\ref{Deguv}. Also, the vertices $u_1$, $u_2$ and $u_3$ form a triangle in $L(G)$. Therefore, by Lemma~\ref{lem:DeltaOrMoreTriangle}, $\Delta(L^k(G))\geq (3 - 2)\cdot 2^{k-1}+2=8\cdot 2^{k-4}+2$.
\end{proof}

\begin{lemma}\label{Deg3toDeg3}
    In a graph $G\in\mathcal G$, if there exists $v_1,v_2\in V(G)$ such that $d_G(v_1),d_G(v_2)\geq 3$, then $\Delta(L^k(G))\geq8\cdot 2^{k-4}+2$ for all $k\geq 2$, i.e. $\dgc(G) \ge 8$.
\end{lemma}

\begin{proof}
    Let $u_1,u_2\in N_G(v_1)\setminus\{v_2\}$ and $u_3,u_4\in N_G(v_2)\setminus\{v_1\}$. Let the vertices in $L(G)$ corresponding to edge $v_1v_2$ be $v$ and to edges $v_1u_1$, $v_1u_2$, $v_2u_3$ and $v_2u_4$ be $v_1'$, $v_2'$, $v_3'$ and $v_4'$ respectively. Now, $d_{L(G)}(v)= d_G(v_1)+d_G(v_2)-2\geq 4$. Moreover, $d_{L(G)}(v_i')=d_G(v_1)+d_G(u_i)-2\geq 2$ for all $i\in\{1,2\}$ and $d_{L(G)}(v_i')=d_G(v_2)+d_G(u_i)-2\geq 2$ for all $i\in\{3,4\}$. Also, $v$ is adjacent to $v_i'$ for all $i\in\{1,2,3,4\}$ as the edge corresponding to vertex $v$ in $G$ shares a vertex each with edges corresponding to vertices $v_i'$, $i\in\{1,2,3,4\}$.

    In $L^2(G)$, let vertices corresponding to edge $vv_i'$ be $s_i$ for all $i\in\{1,2,3,4\}$. So, by Proposition~\ref{Deguv}, $d_{L^2(G)}(s_i)=d_{L(G)}(v)+d_{L(G)}(v_i')-2\geq 4$ for all $i\in\{1,2,3,4\}$. Also, the vertices $s_1$, $s_2$ and $s_3$ form a triangle as their corresponding edges in $L(G)$ are all incident on a common vertex $v$. So, by Lemma~\ref{lem:DeltaOrMoreTriangle}, $\Delta(L^k(G))\geq 2\cdot2^{k-2}+2=8\cdot 2^{k-4}+2$ for all $k\geq 2$.
\end{proof}

\subsection{Structural Results}\label{subsec:structuremaxdeg}

In this section, we focus on the possible values of the degree growth constant for graphs in $\mathcal G$. Our first result shows that the set $\{\dgc(G) \mid G \in \mathcal G\}$ has a minimum, which is $c_1 = 3$.

\begin{lemma}\label{G1M}
    Let $G_1$ be defined as shown in Figure~\ref{G1}. Then $\dgc(G_1) = 3 \le \dgc(G)$ for all $G \in \mathcal G$. Consequently, $\min\{\dgc(G) \mid G \in \mathcal G\} = 3$.
    \begin{figure}[ht!]
    \centering
    \includegraphics[width=0.22\textwidth]{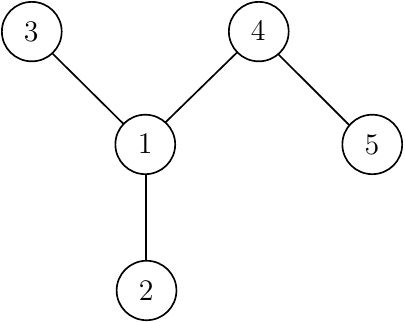}
    \caption{The graph $G_1$}
    \label{G1}
    \end{figure}
\end{lemma}
\begin{proof}
    Note that $L^2(G_1)$ has its $\Delta$-vertices induce a path of two vertices in $L^2(G_1)$ (Figure~\ref{fig:L^2(G_1)}), with $\Delta(L^2(G_1))=3$. Also, $L^2(G_1)$ has $4$ edges connecting a $\Delta$-vertex to a $(\Delta-1)$-vertex. So, by Lemma~\ref{LongDeltaPath}, $L^4(G_1)$ has a $\Delta$-triangle with $\Delta(L^4(G_1))=5$. So, $\Delta(L^k(G_1))=3\cdot2^{k-4}+2$ by Observation~\ref{DirectFormula}, implying $\dgc(G_1) = 3$.
    
    \begin{figure}[ht!]
    \centering
    \includegraphics[width=0.15\textwidth]{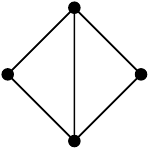}
    \caption{The graph $L^2(G_1)$}
    \label{fig:L^2(G_1)}
    \end{figure}
    
    Now we show that for all graphs $G \in \mathcal G$, $\dgc(G) \ge 3$. We look into the following cases.
    
    \subparagraph*{Case (a): $G$ is isomorphic to a star with $r\ge 4$ leaves ($K_{1, r}$).} The central vertex has degree at least $4$, and therefore Lemma~\ref{Deg4Above} implies $\dgc(G) \ge 8$.

    \subparagraph*{Case (b): $|V(G)| = 4$.} Since $G$ is prolific, we must have $G \in \{CP_4, K_4, K_4 - e\}$, where $K_4 - e$ is the unique graph on $4$ vertices and $5$ edges. Notice that $L(G_1) = CP_4$ and $CP_4 \subseteq K_4 - e \subseteq K_4$, and hence $G_1 \le_{1, 0} G$. By Observation~\ref{obs:r2todgc}, we have $\dgc(G) \ge \dgc(G_1) = 3$.

    \subparagraph*{Case (c): $G$ is not a star $K_{1,r}$ for $r \ge 4$ and has at least $5$ vertices.} In this case, we show that $G_1 \subseteq G$. Since, $G$ is prolific, $\Delta(G) \ge 3$. Let $v$ be a $\Delta$-vertex of $G$. Let $v_1, v_2, v_3$ be three distinct neighbours of $v$. If there exists $u \notin \{v, v_1, v_2, v_3\}$ such that $u$ is adjacent to $v_1$, $v_2$ or $v_3$, we are done as the vertices $u, v, v_1, v_2, v_3$ contain $G_1$ as a subgraph. The only case that remains is when all vertices in $V(G) \setminus \{v\}$ are neighbours of $v$, and no two of them are adjacent. This is a contradiction to $G$ being not a star. Therefore, we have $G_1 \subseteq G$, i.e. $G_1 \le_{0,0} G$ implying $\dgc(G) \ge \dgc(G_1) = 3$.
\end{proof}

Next, we show the existence of the second minimum $c_2 = 4$ of $\{\dgc(G) \mid G \in \mathcal G\}$.

\begin{lemma}\label{G2M}
    Let $G_{2,1,n}$ be a graph in which a pendant vertex is attached to one of the two penultimate vertices of a path with $n-1$ vertices (see Figure~\ref{G21n}). Also, let $G_{2,2,n}$ be a graph in which a pendant vertex each is attached to both of the penultimate vertices of a path with $n-2$ vertices (see Figure~\ref{G22n}). Then, for every integer $a\geq 6$ and $b\geq 8$, we have $\dgc(G_{2,1,a})=\dgc(G_{2,2,b})=4\le\dgc(G)$ for all $G\in\mathcal G\setminus\{G_1\}$. Consequently, $\min\{\dgc(G) \mid G \in \mathcal G\setminus \{G_1\}\} = 4$.
    
    \begin{figure}[ht!]
        \centering
        \includegraphics[width=0.3\textwidth]{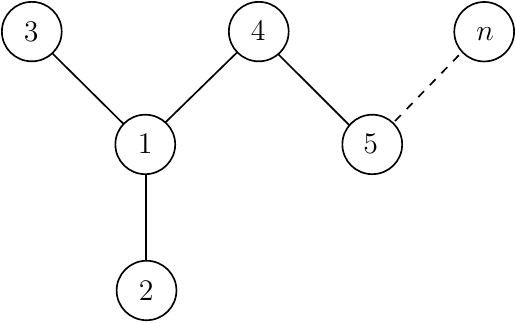}
        \caption{The graph $G_{2,1,n}$. Dashed edge denotes a path of length $\geq 1$}
        \label{G21n}
    \end{figure}
    \begin{figure}[ht!]
        \centering
        \includegraphics[width=0.5\textwidth]{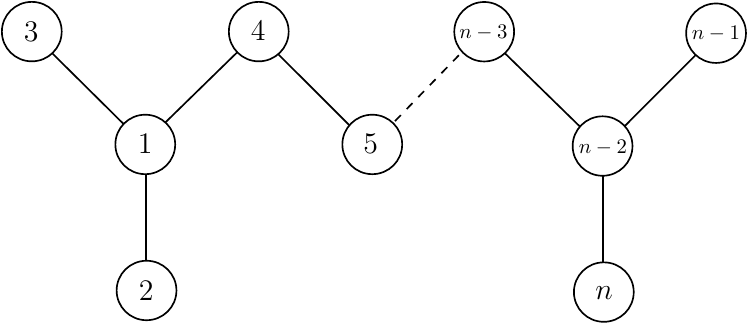}
        \caption{the graph $G_{2,2,n}$. Dashed edge denotes a path of length $\geq 1$}
        \label{G22n}
    \end{figure}
\end{lemma}
\begin{proof}
    All graphs $G_{2,1,a}$ with $a\geq6$ and $G_{2,2,b}$ with $b\geq8$ satisfy $\Delta(L^2(G_{2,1,a}))=\Delta(L^2(G_{2,2,b}))=3$; moreover, $L^2(G_{2,1,a})$ and $L^2(G_{2,2,b})$ have a $\Delta$-triangle (Refer Figure~\ref{L2G2}). So, $\Delta(L^k(G_{2,1,a}))=\Delta(L^k(G_{2,2,b}))=4\cdot2^{k-4}+2$ by Observation~\ref{DirectFormula} for $k\geq 2$. Therefore, we have $\dgc(G_{2,1,a})=\dgc(G_{2,2,b})=4$.

\begin{figure}[ht!]
    \centering
    \includegraphics[width=0.9\textwidth]{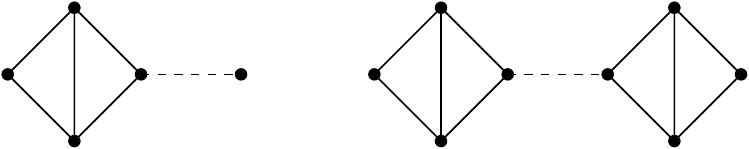}
    \caption{Graphs $L^2(G_{2,1,n}),n\geq6$ and $L^2(G_{2,2,n}),n\geq8$ respectively. Dashed edge denotes a path of length $\geq1$}
    \label{L2G2}
\end{figure}
    Now, we need to show that any graph $G\in\mathcal G\setminus(\{G_1\} \cup \{G_{2,1,n}\mid n\ge 6\} \cup \{G_{2,2,n}\mid n\ge 8\})$ must have $\dgc(G) \geq 4$. 
    
    We have already shown in Lemma~\ref{Deg4Above} and Lemma~\ref{CycPen} that if $G$ is a graph with $\Delta(G)\geq4$ or if it contains a cycle, then $\dgc(G)\geq 6$. It remains to show for acyclic graphs with maximum degree equal to $3$ (recall that $\mathcal{G}$ does not contain graphs with maximum degree at most $2$). Consider an acyclic graph $G\in\mathcal G\setminus\{G_1\}$ satisfying $\Delta(G)=3$. Consider a longest path $P$ with $n$ vertices $v_1,v_2,\ldots,v_n$ of $G$. Notice that there cannot be any other edge connecting two vertices of $P$, otherwise the graph would not be acyclic. Also, $n\ge 4$ as otherwise $G\not\in \mathcal G$. $P$ has the following cases.

    \subparagraph*{Case (a): $n\geq6$.} In this case, if there is a vertex adjacent to any vertex $v_i$, $i\in \{2,,3,\ldots,n-1\}$, then $G_{2,1,6}\leq_{0,0}G$. So, by Observation~\ref{obs:r2todgc}, $4 = \dgc(G_{2,1,6})\le\dgc(G)$.

    \subparagraph*{Case (b): $n=5$.} In this case, if there is a vertex $u\not\in V(P)$ adjacent to any vertex $v_i$, $i\in \{2,4\}$, then $G_{2,1,6}\leq_{0,0}G$. So, by Observation~\ref{obs:r2todgc}, $4 = \dgc(G_{2,1,6})\le \dgc(G)$. Otherwise, there is at least one vertex adjacent to $v_3$. So, $G_3\leq_{0,0}G$ (see Figure~\ref{fig:G3}) and $G_{2,1,6}\leq_{2,2}G_3$ (see Figure~\ref{L^(G_3)-vs-L^2(G_{2,1,6})}). So, by Observation~\ref{obs:r2todgc}, $4 = \dgc(G_{2,1,6})\le \dgc(G_3)\le \dgc(G)$.

    \begin{figure} [!htbp]
        \begin{minipage}[t]{0.49\textwidth}
            \centering
            \includegraphics[width = 0.5\textwidth]{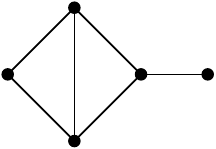}
            \subcaption{The graph $L^2(G_{2,1,6})$.} \label{fig:L^2(G_{2,1,6})}
        \end{minipage} \hfill \hfill 
        \begin{minipage}[t]{0.49\textwidth}  
            \centering 
            \includegraphics[width = 0.65\textwidth]{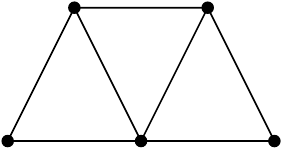}
            \subcaption{The graph $L^2(G_3)$.} \label{fig:L^2(G_3)}
        \end{minipage} 
        \caption{Graphs $L^2(G_3)$ and $L^2(G_{2,1,6})$.}\label{L^(G_3)-vs-L^2(G_{2,1,6})}
    \end{figure}

    \subparagraph*{Case (c): $n=4$.} In this case, there can be a path of at most one vertex outside $P$ attached to either $v_2$ or $v_3$, as otherwise $P$ would not be a longest path of $G$. If the vertex is adjacent to only $v_2$ or $v_3$, then $G$ is isomorphic to $G_1$. Otherwise, if there is one vertex each adjacent to $v_2$ and $v_3$, then $G$ has a degree $3$ vertex $v_2$ adjacent to another degree $3$ vertex $v_3$. So, by Lemma~\ref{Deg3toDeg3}, $\dgc(G)=8$.

    Above cases are exhaustive. So, for any graph $G\in G\setminus\{G_1\}$, $\dgc(G)\ge 4$.
\end{proof}

Moreover, as a consequence of Lemma~\ref{G2M}, we get that $G_1$ is the only prolific graph with $\dgc(G) = 3$, i.e. $\mathcal G_1=\{G \in \mathcal G \mid \dgc(G) = c_1\} = \{G_1\}$. For completeness we state this as a theorem below.

\begin{theorem}
    $\mathcal G_1=\{G \in \mathcal G \mid \dgc(G) = c_1\} = \{G_1\}$.
\end{theorem}

We now show that $\{ \dgc(G) \mid G \in \mathcal G \}$ has a third minimum, $c_3 = 5.5$. It is interesting to note that unlike the first and the second minimum, the third minimum is not an integer.

\begin{lemma}\label{G3M}
    Let $G_3$ be defined as shown in Figure~\ref{fig:G3}. Then $\dgc(G_3) = 5.5 \le \dgc(G)$ for all $G \in \mathcal G\setminus(\{G_1\} \cup \{G_{2,1,n}\mid n\ge 6\} \cup \{G_{2,2,n}\mid n\ge 8\})$. Consequently, $\min\{\dgc(G) \mid G \in \mathcal G\setminus(\{G_1\} \cup \{G_{2,1,n}\mid n\ge 6\} \cup \{G_{2,2,n}\mid n\ge 8\})\} = 5.5$.
    
    \begin{figure}[ht!]
        \centering
        \includegraphics[width=0.3\textwidth]{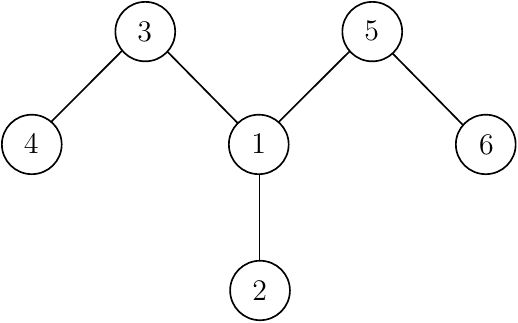}
        \caption{The graph $G_{3}$}
        \label{fig:G3}
    \end{figure}
\end{lemma}

\begin{proof}
    The graph $L^3(G_3)$ (refer Figure~\ref{L3G3}) has its $\Delta$-vertices induce a path of $2$ vertices in $L^3(G_3)$ and moreover there exists at least three edges connecting a $\Delta$-vertex to a $(\Delta-1)$-vertex of $L^3(G_3)$. Note that, $\Delta(L^3(G_3)))=5$. So, $L^5(G_3)$ contains a $\Delta$-triangle with $\Delta(L^5(G_3))=13$ by Lemma~\ref{LongDeltaPath}. Hence, $\Delta(L^k(G_3))=11\cdot2^{k-5}+2=5.5\cdot2^{k-4}+2$, for all $k\geq5$ by Observation~\ref{DirectFormula}. Therefore, we have $\dgc(G_3)=5.5$.

    \begin{figure}[ht!]
        \centering
        \includegraphics[width=0.3\textwidth]{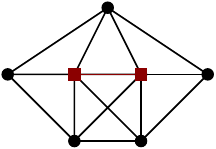}
        \caption{Graph $L^3(G_{3})$, square vertices are the $\Delta$-vertices with degree $5$ that induce a path of length $2$}
        \label{L3G3}
    \end{figure}
    Now, we need to show that any graph $G\in\mathcal G\setminus(\{G_1\} \cup \{G_{2,1,n}\mid n\ge 6\} \cup \{G_{2,2,n}\mid n\ge 8\} \cup\{G_3\})$ must have $\dgc(G) \geq 5.5$. 
    
    We have already shown in Lemma~\ref{Deg4Above} and Lemma~\ref{CycPen} that if $G$ is a graph with $\Delta(G)\geq4$ or if it contains a cycle, then $\dgc(G)\geq 6$. It remains to show for acyclic graphs with maximum degree equal to $3$ (recall that $\mathcal{G}$ does not contain graphs with maximum degree at most $2$). Consider an acyclic graph $G\in\mathcal G\setminus(\{G_1\} \cup \{G_{2,1,n}\mid n\ge 6\} \cup \{G_{2,2,n}\mid n\ge 8\} \cup\{G_3\})$ satisfying $\Delta(G)=3$. Consider a longest path $P$ with $n$ vertices $v_1,v_2,\ldots,v_n$ of $G$. Notice that there cannot be any other edge connecting two vertices of $P$, otherwise the graph would not be acyclic. Also, $n\ge 4$ as otherwise $G\not\in \mathcal G$. $G$ has the following cases.

    \subparagraph*{Case (a): $n\geq6$.} In this case, if there is a vertex adjacent to any vertex $v_i$, $i\in \{3,4,\ldots,n-2\}$, then $G_{3}\leq_{0,0}G$. So, by Observation~\ref{obs:r2todgc}, $\dgc(G_{3})\le\dgc(G)$. Otherwise, there can be a path of at most one vertex outside $P$ incident at either $v_2$ or $v_{n-1}$ or both, as otherwise $P$ would not be a longest path of $G$. If a vertex is adjacent to only $v_2$ or $v_{n-1}$ or if a vertex each is adjacent to $v_2$ and $v_{n-1}$, then $G=G_{2,1,n+1}$ or $G=G_{2,2,n+2}$, which is a contradiction.

    \subparagraph*{Case (b): $n=5$.} In this case, if there is a vertex $u\not\in V(P)$ adjacent to $v_3$, then $G_{3}\leq_{0,0}G$. So, by Observation~\ref{obs:r2todgc}, $\dgc(G_{3})\le \dgc(G)$. Otherwise, there can be a path of at most one vertex outside $P$ incident at either $v_2$ or $v_4$ or both, as otherwise $P$ would not be a longest path of $G$. If a vertex is adjacent to only $v_2$ or $v_4$ then $G=G_{2,1,6}$ --- a contradiction. Otherwise, there is a vertex each adjacent to $v_2$ and $v_4$. So, $G=G_{2,2,7}$. Therefore, $L^3(G)$ has a $\Delta$-triangle with $\Delta(G) = 5$ (refer to Figure~\ref{fig:L^3(G_{2,2,7}}) and by Observation~\ref{DirectFormula}, we have $\Delta(L^k(G)) = (5 - 2) \cdot 2^{k-3} + 2 = 6 \cdot 2^{k -4} + 2$ for $k \ge 3$. So, by Observation~\ref{obs:r2todgc}, $5.5 = \dgc(G_3)\le \dgc(G) = 6$.

    \begin{figure}
        \centering
        \includegraphics[width=0.5\linewidth]{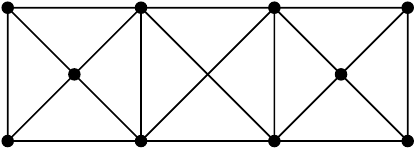}
        \caption{The graph $L^3(G_{2,2,7})$.}
        \label{fig:L^3(G_{2,2,7}}
    \end{figure}

    \subparagraph*{Case (c): $n=4$.} In this case, there can be a path of at most one vertex outside $P$ incident at either $v_2$ or $v_3$, as otherwise $P$ would not be a longest path of $G$. If the singular vertex is adjacent to only $v_2$ or $v_3$, then $G$ is isomorphic to $G_1$. Otherwise, if there is one vertex each adjacent to $v_2$ and $v_3$, then $G$ has a degree $3$ vertex $v_2$ adjacent to another degree $3$ vertex $v_3$. So, by Lemma~\ref{Deg3toDeg3}, $\dgc(G)=8$.

    The above cases are exhaustive. So, for any graph $G\in G\setminus(\{G_1\} \cup \{G_{2,1,n}\mid n\ge 6\} \cup \{G_{2,2,n}\mid n\ge 8\})$, $\dgc(G)\ge 5.5$.
\end{proof}

This gives us the set of graphs attaining the second minimum among all degree growth constants of prolific graphs, i.e. $\mathcal G_2 = \{G \in \mathcal G \mid \dgc(G) = c_2\}$.
\begin{theorem}\label{G2G}
    $\mathcal G_2 = \{G_{2,1,n} \mid n\geq6\}\cup\{G_{2,2,n} \mid n\geq8\} = \{G \in \mathcal G \mid \dgc(G) = c_2\}$.
\end{theorem}
\begin{proof}
    Every graph in $\mathcal G\setminus(\{G_1\}\cup\{G_{2,1,n} \mid n\geq6\}\cup\{G_{2,2,n} \mid n\geq8\})$ satisfies $\dgc(G) \ge 5.5$ as shown in Lemma \ref{G3M}. Also, by Lemma \ref{G2M}, $G_{2,1,n}\in\mathcal G_2$ for all $n\geq 6$ and $G_{2,2,n}\in\mathcal G_2$ for all $n\geq 8$ implying $\mathcal G_2=\{G_{2,1,n} \mid n\geq6\}\cup\{G_{2,2,n} \mid n\geq8\}$.
\end{proof}

Next, we study the fourth minimum among all degree growth constants. We show that the fourth minimum is $c_3 = 6$, which is again an integer, unlike the third minimum. We start by defining the following. 

\begin{definition}\label{def:H_4}
    Let $\mathcal{H}_4'$ be a set of graphs where each graph contains (i) a cycle or a path $\tilde{H}$, (ii) every remaining vertex is a pendant vertex with its neighbour in $\tilde{H}$, and (iii) the distance between any two degree $3$ vertices is at least $3$. Let $\mathcal H_4=\mathcal{H}_4'\setminus(\mathcal G_1\cup \mathcal G_2 \cup \{G_3\})$ (Figure~\ref{H4}).
\end{definition}


\begin{figure} [!htbp]
    \centering 
    \includegraphics[width=0.9\textwidth]{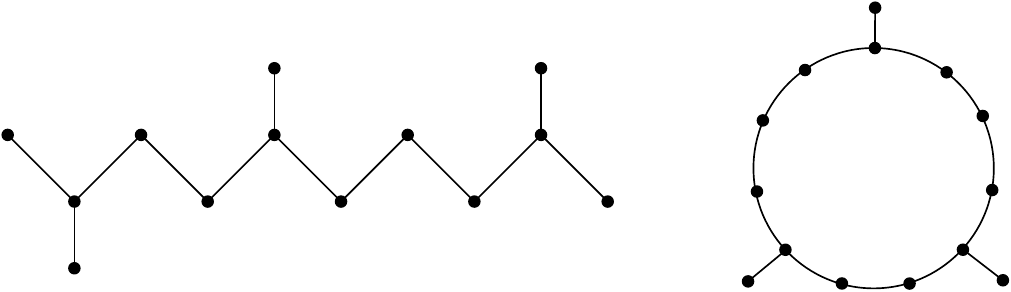}
    \caption{Example of graphs in the set $\mathcal{H}_4$.}\label{H4}
\end{figure}

\begin{figure} [!htbp]
    \centering
    \includegraphics[width=0.9\textwidth]{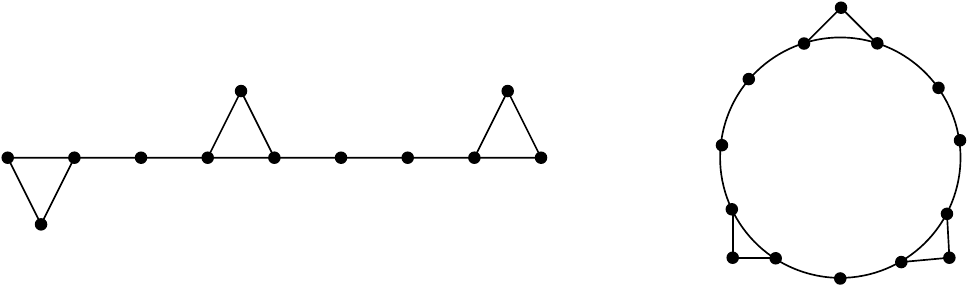}
    \caption{Example of line graphs of graphs in the set $\mathcal{H}_4$.}\label{L(H4)}
\end{figure}

\begin{theorem}\label{G4M}
    Let $G_4$ be defined as shown in Figure~\ref{G4}. Then, for all $H\in\mathcal H_4$, we have $\dgc(G_4) = \dgc(H) = 6$. Moreover, $7\le \dgc(G)$ for all $G \in \mathcal G\setminus(\mathcal G_1 \cup \mathcal G_2\cup\{G_3\}\cup\{G_4\}\cup\mathcal H_4)$. Consequently, $\min\{\dgc(G) \mid G \in \mathcal G\setminus(\mathcal G_1 \cup \mathcal G_2\cup \{G_3\})\} = 6$.
    \begin{figure}[ht!]
        \centering
        \includegraphics[width=0.3\textwidth]{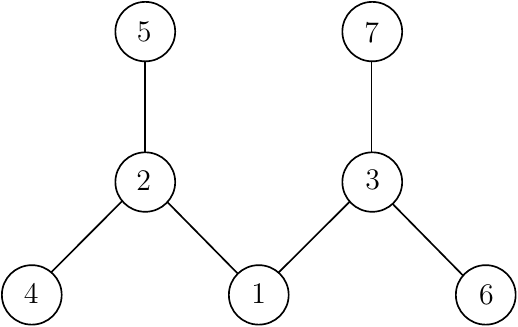}
        \caption{The graph $G_{4}$}
        \label{G4}
    \end{figure}
\end{theorem}
    
\begin{proof}
    Consider the graph $L(G)$ for a graph $G\in\mathcal H_4\cup \{G_4\}$. If $G=G_4$, then $L(G)$ has its $\Delta$-vertices induce a path of $2$ vertices in $L(G)$ with $\Delta(L(G))=3$, and there are $4$ edges connecting a $\Delta$-vertex to a $(\Delta-1)$-vertex in $L(G)$ (Figure~\ref{fig:L(G_4)}). So, by Lemma~\ref{LongDeltaPath}, $L^3(G)$ has a $\Delta$-triangle with $\Delta(L^3(G))=5$. Otherwise, $G\in \mathcal H_4$. Then we show the following.

    \begin{figure}
        \centering
        \includegraphics[width=0.3\linewidth]{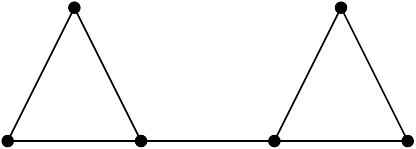}
        \caption{The graph $L(G_4)$.}
        \label{fig:L(G_4)}
    \end{figure}
    
    \begin{claim}
        For every $G \in \mathcal H_4$, we have $\Delta(L^3(G)) = 5$ and $L^3(G)$ has a $\Delta$-cycle.
    \end{claim}
    \begin{claimproof}
        Since $G \in \mathcal H_4$, $G$ contains the following: (i) a cycle or a path $\tilde{H}$ containing the vertices in order $v_1, v_2, \ldots, v_n$, (ii) every remaining vertex is a pendant vertex with its neighbour in $\tilde{H}$, and (iii) the distance between any two degree $3$ vertices is at least $3$. If $\tilde H$ is a cycle, we may assume without loss of generality that degrees of $v_1$ and $v_n$ are $2$ (otherwise we can rotate the labels of the cycle to satisfy this condition), and we define $v_{n+1} = v_1$ and $v_0 = v_n$. If $\tilde{H}$ is a path, then we may assume $v_1$ and $v_n$ to have degrees $1$, since otherwise we have a longer choice of $\tilde H$ such that the endpoints of $\tilde H$ have degree $1$. Let the pendant vertices be $u_{a_1}, u_{a_2}, \ldots, u_{a_t}$, such that (i) $2 \le a_1$, (ii) $a_t \le (n-1)$, (iii) $a_{i+1} - a_{i} \ge 3$ for $i \in \{1, 2, \ldots, t - 1\}$, and (iv) $u_{a_i}$ is adjacent to $v_{a_i}$ for $i \in \{1, 2, \ldots, t\}$.

        We now analyze the structure of $L(G)$ (Figure~\ref{L(H4)}). In $L(G)$, the vertices $x_1, x_2, \ldots, x_{n-1}$, corresponding to the edges $v_1v_2, v_2v_3, \ldots, v_{n}v_{n-1} \in E(G)$ induce a path. Moreover, if $\tilde H$ is a cycle, then there exists $x_n \in V(L(G))$ adjacent to both $x_1$ and $x_{n-1}$, corresponding to the edge $v_nv_1$. For $i \in \{1, \ldots, t\}$, $L(G)$ contains the vertices $y_{a_i}$ corresponding to the edge $v_{a_i}u_{a_i} \in E(G)$. $y_{a_i}$ is adjacent to $x_{a_i}$ and $x_{a_i + 1}$. This completes the description of $L(G)$. Notice that $\Delta(L(G)) = 3$. 
        
        If $\tilde H$ is a cycle, then the $\Delta$-vertices are precisely $x_{a_1}, x_{a_1 + 1}, x_{a_2}, x_{a_2 + 1}, \ldots, x_{a_t}, x_{a_t + 1}$. Moreover $x_{a_1 - 1}, x_{a_2 - 1}, \ldots x_{a_t - 1}$ are degree $2$ vertices. 
        
        If $\tilde H$ is a path, then the $\Delta$-vertices are precisely $x_{a_i}$ for every $i$ such that $a_i \ge 3$ and $x_{a_i + 1}$ for every $i$ such that $a_i \le n - 2$. In such a case, there always exists an $i$ such that $3 \le a_i \le n - 2$ as $G \notin \mathcal G_2$. Therefore there exists $i$ such that $x_{a_i}$ and $x_{a_i +1}$ are both $\Delta$-vertices of $L(G)$. Further, as $G \ne G_3$, either $x_{a_i - 1}$ or $x_{a_i + 2}$ is a degree $2$ vertex. 

        From these two paragraphs, we get that there exists $i$ such that $x_{a_i}$ and $x_{a_i + 1}$ are $\Delta$-vertices of $L(G)$ with degree $3$ and there exists vertex $x' \in \{x_{a_i - 1}, x_{a_i + 2}\}$ of degree $2$. Let $e$ be the edge between $x'$ and one of $x_{a_i}$ or $x_{a_i + 1}$. Firstly this implies $\Delta(L^2(G)) = 2\Delta(L(G)) - 2 = 4$ by Proposition~\ref{Deguv} where the vertex $z \in V(L^2(G))$ corresponding to $x_{a_i}x_{a_i + 1} \in E(L(G))$ is a $\Delta$-vertex. Secondly, the vertices $p, q, r \in V(L^2(G))$ corresponding to $x_{a_i}y_{a_i}, x_{a_i+1}y_{a_i}, e \in E(L(G))$, are $(\Delta-1)$-vertices having degree $3$ (by Proposition~\ref{Deguv}) and are adjacent to $z$. Finally no two $\Delta$-vertices of $L^2(G)$ are adjacent, as $a_{i+1} - a_i \ge 3$ for all $i \in \{1, \ldots, t-1\}$.

        These imply that $\Delta(L^3(G)) \le \Delta(L^2(G)) + (\Delta(L^2(G)) - 1)- 2 = 5$ (by Proposition~\ref{Deguv}). However, the vertices $\alpha, \beta, \gamma \in V(L^3(G))$ corresponding to the edges $pz, qz, rz \in E(L^2(G))$ have degree $5$ (by Proposition~\ref{Deguv}). Therefore, $\alpha, \beta, \gamma$ form a $\Delta$-triangle of $L^3(G)$ with $\Delta(L^3(G)) = 5$.
    \end{claimproof}
    
    So, for $G \in \{G_4\} \cup \mathcal H_4$, $L^3(G)$ has a $\Delta$-triangle with $\Delta(L^3(G))=5$. Hence, $\Delta(L^k(G))=3\cdot2^{k-3}+2=6\cdot2^{k-4}+2$ for all $k\geq3$ by Observation \ref{DirectFormula}. Therefore, for $G_4$ and for every $H\in\mathcal H_4$, we have $\dgc(G_4)=\dgc(H)=6$. All that remains is to prove the following claim.

    \begin{claim}\label{clm:ge7}
        For every graph $G\in\mathcal G\setminus(\{G_1\} \cup \{G_{2,1,n}\mid n\ge 6\} \cup \{G_{2,2,n}\mid n\ge 8\} \cup\{G_3\} \cup \{G_4\}\cup\mathcal H_4)$, $\dgc(G)\ge 7$.
    \end{claim}

    \begin{claimproof}
        We have already shown in Lemma~\ref{Deg4Above}, Lemma~\ref{Deg3to3deg2} and Lemma~\ref{Deg3toDeg3}, that if $G$  is a graph satisfying one of the following conditions, (i) $\Delta(G)\geq4$, (ii) $G$ contains a degree $3$ vertex adjacent to three degree $\geq 2$ vertices or (iii) $G$ contains a degree $3$ vertex adjacent to a degree $3$ vertex, then $\dgc(G) \ge 8$.
    
        Now we only look into graphs $G$ not satisfying (i), (ii) and (iii). Since $G$ does not satisfy (i) and (ii), it must be of the following form: $G$ contains a path or a cycle $\tilde{H}$ and one or more pendant vertices adjacent to distinct vertices of $\tilde{H}$. If $G$ has only one pendant vertex, or if the distance between any two degree $3$ vertices of $G$ is at least than $3$, then $G\in\mathcal H_4$. Otherwise, if the distance between any two degree $3$ vertices of $G$ is equal to $1$, then $G$ falls in condition (iii). So, it means there exist two vertices of degree $3$ vertices of $G$ with distance exactly $2$ between them. If $\tilde{H}$ is a cycle of length $4$, then $G$ isomorphic to $H_{s}$ (Figure~\ref{fig:H_s}). However, $L(H_s)$ has a $\Delta$-cycle where $\Delta(L(H_{s})) = 3$ (Figure~\ref{fig:L(H_s)}) and by Lemma~\ref{DirectFormula} we have $\Delta(L^k(G)) = (3 - 2) \cdot 2^{k-1} + 2 = 8 \cdot 2^{k-4} + 2$ for $k \ge 1$, i.e. $\dgc(H_{s}) = 8$. Otherwise, if $\tilde{H}$ is a path or a cycle of at least $5$ vertices, we have $G_4 \subseteq G$. We now analyze the following cases for $G_4 \subseteq G$.

    \begin{figure} [!htbp]
        \begin{minipage}[t]{0.49\textwidth}  
            \centering 
            \includegraphics[width = 0.7\textwidth]{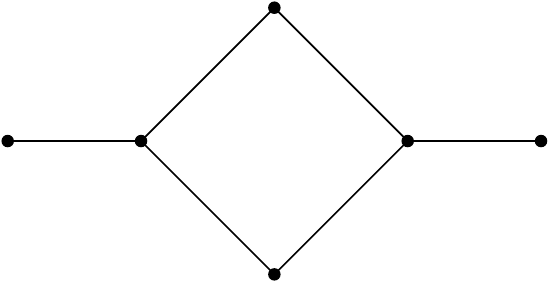}
            \subcaption{The graph $H_s$.} \label{fig:H_s}
        \end{minipage} \hfill \hfill 
        \begin{minipage}[t]{0.49\textwidth}
            \centering
            \includegraphics[width = 0.7\textwidth]{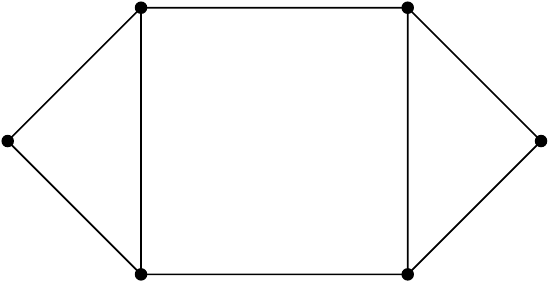}
            \subcaption{The graph $L(H_s)$.} \label{fig:L(H_s)}
        \end{minipage}
        \caption{Graphs $H_s$ and $L(H_s)$.}
    \end{figure}
        
        Consider constructing $G$ by adding vertices and edges to an initial copy of $G_4$; we call the edges and vertices not appearing in the initial copy of $G_4$ as additional edges and additional vertices, respectively. Let the vertices of $G_4$ be labeled as in Figure~\ref{G4}. Now, if we add an edge and/or vertex to $G_4$, it can lead to the following cases:
        
            \subparagraph*{Case (a): An additional vertex is adjacent to one of the four degree $1$ vertices (i.e. vertices labeled $4, 5, 6$ and $7$ in Figure~\ref{G4}).} Let $G_4'$ be the graph obtained by adding this additional vertex to one of the degree $1$ vertices in $G_4$, then $G_4' \subseteq G$ (Figure~\ref{fig:G_4'}). $\Delta(L(G'_4))=3$ and $L(G'_4)$ has its $\Delta$-vertices induce a path of $3$ vertices with at least $3$ edges connecting a $\Delta$-vertex to a $(\Delta-1)$-vertex (Figure~\ref{fig:L(G_4')}). So, by Lemma~\ref{LongDeltaPath}, $L^4(G'_4)$ has a $\Delta$-triangle with $\Delta(L^4(G_4'))=9$. So, by Observation~\ref{DirectFormula}, $\Delta(L^k(G_4'))=7\cdot2^{k-4}+2$ for all $k\geq 4$, therefore $\dgc(G) \ge \dgc(G_4')= 7$.
            \subparagraph*{Case (b): An additional edge is adjacent to one of the $2$ degree $3$ vertices (i.e. vertices labelled $2$ and $3$).} In this case, resulting graph has maximum degree equal to $4$, and hence by Lemma~\ref{Deg4Above}, $\dgc(G)\ge 8$.
            \subparagraph*{Case (c): An additional edge is adjacent to the degree $2$ vertex (i.e. vertex labelled $1$).} In this case, resulting graph has a degree $3$ vertex adjacent to another degree $3$ vertex, and hence by Lemma~\ref{Deg3toDeg3}, $\dgc(G)\ge 8$.
            \subparagraph*{Case (d): An additional edge connects a degree $1$ vertex to another degree $1$ vertex adjacent to the same degree $3$ vertex (i.e. between vertices labelled $4$ and $5$ or $6$ and $7$).} In this case, resulting graph has a degree $3$ vertex adjacent to $3$ degree $2$ vertices, and hence by Lemma~\ref{Deg3to3deg2}, $\dgc(G)\ge 8$.
            \subparagraph*{Case (e): An additional edge connects a degree $1$ vertex to another degree $1$ vertex which in turn is adjacent to the other degree $3$ vertex (i.e. between vertices labelled $4$ and $6$, $4$ and $7$, $5$ and $6$ or $5$ and $7$).} Let $G_4''$ be the graph obtained by adding this additional edge between two of the degree $1$ vertices which in turn are adjacent to distinct degree $3$ vertices in $G_4$ (Figure~\ref{fig:G_4''}), then $G_4''\subseteq G$. $\Delta(L(G''_4))=3$ and $L(G_4'')$ has its $\Delta$-vertices induce a path of length $4$ in $L(G_4'')$ with at least $3$ edges connecting a $\Delta$-vertex to a $(\Delta-1)$-vertex (Figure~\ref{fig:L(G_4'')}). So, by Lemma~\ref{LongDeltaPath}, $L^5(G''_4)$ has a $\Delta$-triangle with $\Delta(L^5(G_4''))=17$. So, by Observation~\ref{DirectFormula}, $\Delta(L^k(G_4''))=7.5\cdot2^{k-4}+2$ for all $k\geq 5$ and hence, $\dgc(G)\ge\dgc(G_4'')=7.5$.

    \begin{figure} [!htbp]
        \begin{minipage}[t]{0.49\textwidth}  
            \centering 
            \includegraphics[width = 0.7\textwidth]{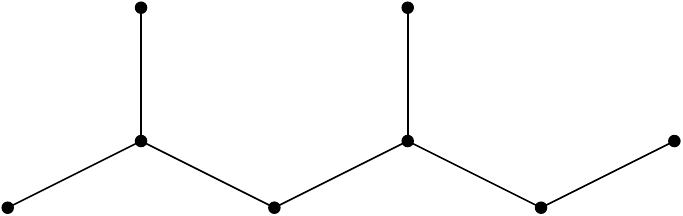}
            \subcaption{The graph $G_4'$.} \label{fig:G_4'}
        \end{minipage} \hfill \hfill 
        \begin{minipage}[t]{0.49\textwidth}
            \centering
            \includegraphics[width = 0.7\textwidth]{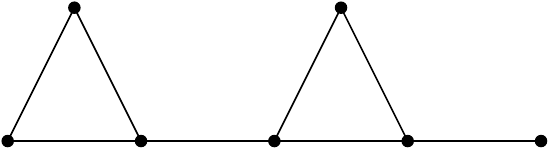}
            \subcaption{The graph $L(G_4')$.} \label{fig:L(G_4')}
        \end{minipage} \\
        \begin{minipage}[t]{0.49\textwidth} 
            \centering 
            \includegraphics[width = 0.5\textwidth]{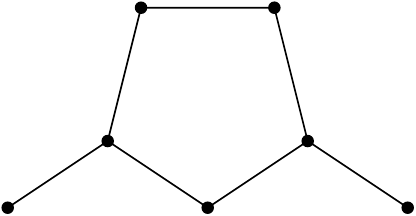}
            \subcaption{The graph $G_4''$.} \label{fig:G_4''}
        \end{minipage} \hfill \hfill 
        \begin{minipage}[t]{0.49\textwidth}
            \centering
            \includegraphics[width = 0.5\textwidth]{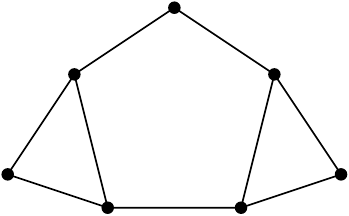}
            \subcaption{The graph $L(G_4'')$.} \label{fig:L(G_4'')}
        \end{minipage}
        \caption{Graphs $G_4'$, $L(G_4')$, $G_4''$ and $L(G_4'')$}
    \end{figure}

        So, adding any edge and/or vertex to graph $G_4$ leads to the resulting graph $G_4'$ having $\Delta(L^k(G_4'))>6\cdot 2^{k-4}+2$ for all $k\geq 5$. Since any graph $G$ containing $G_4$ as a proper subgraph will contain one of the above cases as a subgraph. Hence, by Lemma~\ref{Comp1}, $\Delta(L^k(G))>6\cdot 2^{k-4}+2$.
    \end{claimproof}

    So, we have proven the claim. Hence, $\dgc(G)\ge 7$ for all $G\in\mathcal G\setminus(\{G_1\} \cup \{G_{2,1,n}\mid n\ge 6\} \cup \{G_{2,2,n}\mid n\ge 8\} \cup\{G_3\} \cup \{G_4\}\cup\mathcal H_4)$.
\end{proof}

Not only does this prove that the fourth minimum of $\{\dgc(G) \mid G \in \mathcal G\}$ is $c_4 = 6$, this shows, that $\{\dgc(G) \mid G \in \mathcal G, 6 < \dgc(G) < 7\} = \emptyset$. Moreover, the proof also implies the exact classification of graphs $\mathcal G_3$ and $\mathcal{G_4}$ attaining $\dgc(G) = c_3$ and $\dgc(G) = c_4$ respectively. We state this formally as follows.

\begin{theorem}\label{thm:G3}
    $\mathcal G_3= \{G \mid G \in \mathcal G, \dgc(G) = c_3\} = \{G_3\}$ and $\mathcal G_4= \{G \mid G \in \mathcal G, \dgc(G) = c_4\} = \{G_4\} \cup \mathcal H_4$.
\end{theorem}
\begin{proof} 
    Every graph in $\mathcal G\setminus(\{G_1\}\cup\{G_{2,1,n} \mid n\geq6\}\cup\{G_{2,2,n}\mid n\geq8\}\cup G_3)$ satisfies $\dgc(G) \ge 6$; in particular every graph $G\in\mathcal G\setminus(\{G_1\} \cup \{G_{2,1,n}\mid n\ge 6\} \cup \{G_{2,2,n}\mid n\ge 8\} \cup\{G_3\} \cup \{G_4\}\cup\mathcal H_4)$, has $\dgc(G)\ge 7$ as shown in Theorem~\ref{G4M}. Therefore $\mathcal G_3 = \{G_3\} = \{G \mid G \in \mathcal G, \dgc(G) = c_3\}$ by Lemma~\ref{G3M} and $\mathcal G_4 = \{G_4\} \cup \mathcal H_4= \{G \mid G \in \mathcal G, \dgc(G) = c_4\}$ by Theorem~\ref{G4M}.
\end{proof}

Theorem~\ref{G4M} shows that $\dgc(G) \ge 7$ for all $G \in \mathcal G \setminus (\mathcal G_1 \cup \mathcal G_2 \cup \mathcal G_3 \cup \mathcal G_4)$. We also show that indeed $c_5 = 7$ is the fifth minimum by exhibiting a graph attaining $\dgc(G) = 7$ (Please refer to the proof of Theorem~\ref{G4M}). 
We restate this result formally as the following for completeness. 

\begin{figure}
    \centering
    \includegraphics[width=0.35\linewidth]{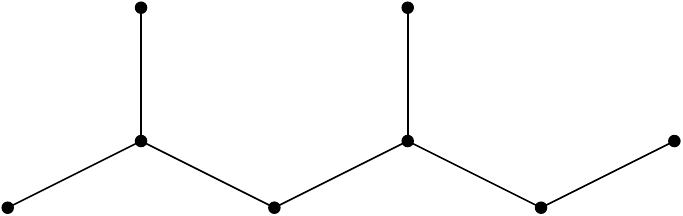}
    \caption{The graph $G_4'$.}
    \label{fig:G_4'-main}
\end{figure}
\begin{theorem}
    Let $G_4'$ be as defined in Figure~\ref{fig:G_4'-main}. Then, $\dgc(G_4') = 7$. Consequently, $\min\{\dgc(G) \mid G \in \mathcal G \setminus (\mathcal G_1 \cup \mathcal G_2 \cup \mathcal G_3 \cup \mathcal G_4)\} = 7$.
\end{theorem}

Therefore, we deduce the first five minimums $c_1 = 3$, $c_2 = 4$, $c_3 = 5.5$, $c_4 = 6$, and $c_5 = 7$ of $\{\dgc(G) \mid G \in \mathcal G\}$ and derive the precise classes of graphs $\mathcal G_1$, $\mathcal G_2$, $\mathcal G_3$ and $\mathcal G_4$, attaining the first four minimum degree growth constants. Such a result might motivate us to investigate whether the sixth minimum would be $8$. We answer this negatively by showing that there are in fact infinitely many distinct values of $\dgc(G)$ between $7$ and $8$ realized by prolific graphs.

\begin{theorem}\label{infcases}
    Let $L_n$ be a graph in which a pendant is adjacent to every alternate vertex of a path of $2n+1$ vertices for $n\geq 3$ (shown in Figure~\ref{Ln}). We have $\Delta(L^k(L_n))=\big(8-\frac{1}{2^{2n-5}}\big)\cdot2^{k-4}+2$ for all $k\geq 2n-1$. So, $\dgc(L_n)=8-\frac{1}{2^{2n-5}}$, and hence there are infinitely many possible values of $\dgc(G)$, $G\in\mathcal G$ between $7$ and $8$.

    \begin{figure}[ht!]
        \centering
        \includegraphics[width=0.9\textwidth]{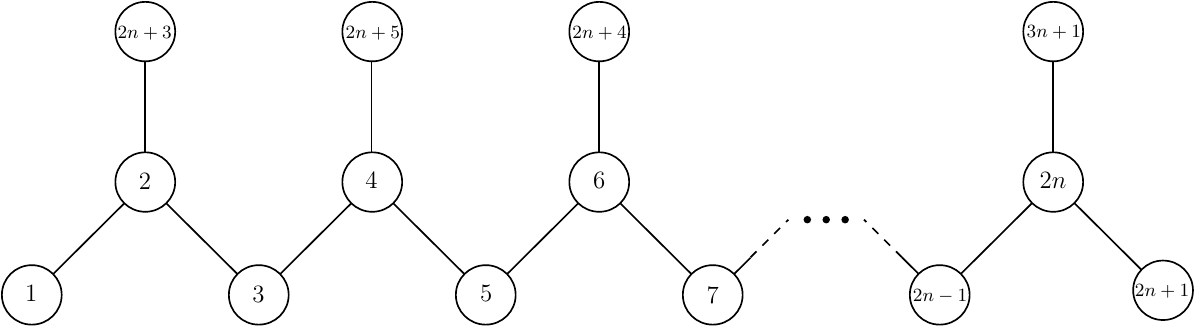}
        \caption{Graph $L_n$}
        \label{Ln}
    \end{figure}
\end{theorem}

\begin{proof}   
    The graph $L(L_n)$ (Shown in Figure~\ref{LLn}) has $(2n-2)$ $\Delta$-vertices that induce a path in $L(L_n)$. Also, there are $2n$ edges that connect a $\Delta$-vertex to a $(\Delta-1)$-vertex in $L(L_n)$. So, for $n\geq 3$, by Lemma~\ref{LongDeltaPath}, $L^{2n-2}(L(L_n))=L^{2n-1}(L_n)$ has a $\Delta$-triangle with $\Delta(L^{2n-1}(L_n))=2^{2n-2}+1$. So, by Observation~\ref{DirectFormula}, we have $\Delta(L^k(G))=(2^{2n-2}-1)\cdot 2^{k-2n+1}+2=\big(8-\frac{1}{2^{2n-5}}\big)\cdot2^{k-4}+2$, and therefore $\dgc(L_n)=8-\frac{1}{2^{2n-5}}$.
    \begin{figure}[ht!]
        \centering
        \includegraphics[width=0.9\textwidth]{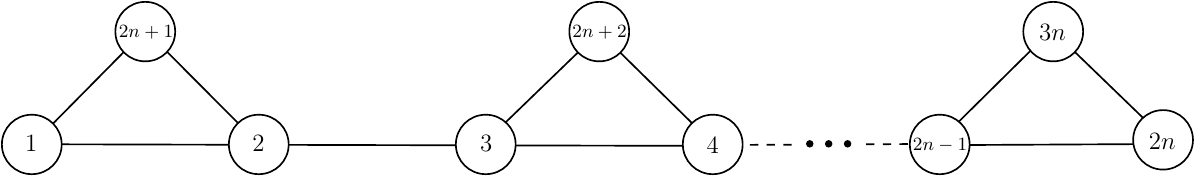}
        \caption{Graph $L(L_n)$}
        \label{LLn}
    \end{figure}
\end{proof}

\section{Open Questions}


 This paper examines two structural properties of iterated line graphs: (i) existence of Euler paths and (ii) growth of maximum degree. In future, we hope to extend the understanding of iterated higher order line graphs with respect to other graph properties that are efficiently checkable, e.g., $3$-connectedness, etc. 
 Note that we characterize the precise graphs attaining the first four minimum values of the degree growth constant, but it remains open to obtain the precise graphs attaining the fifth minimum of $7$. An efficient algorithm to compute the $\dgc(G)$ of an input prolific graph $G$ remains an open problem. Another open algorithmic question would be to design efficient algorithms (or provide lower bounds) to compute the maximum degree of an iterated higher order line graph for some input graph $G$. 

\bibliography{Source_Files/Bibliography}
\end{document}